\documentclass[3p,times]{elsarticle}    
\usepackage[english]{babel} \usepackage[latin1]{inputenc}       \usepackage[T1]{fontenc}
\usepackage{amsmath,amsthm,amsfonts,amssymb,latexsym,graphics,wrapfig,enumerate}
\usepackage{dsfont,times,verbatim,mathrsfs,calrsfs,times,microtype}
\usepackage[matrix,arrow,curve,frame,arc,line]{xy} 
\usepackage[d]{esvect}
\usepackage{pstricks-add}
\usepackage{graphicx}

\newtheorem*{lemma}{Lemma}
\newtheorem*{proposition}{Proposition}
\newtheorem*{thm}{Theorem}
\newtheorem*{corollary}{Corollary}

\theoremstyle{remark}
    \newtheorem*{remark}{Remark}

\theoremstyle{definition}
    \newtheorem*{definition}{Definition}
    \newcommand{\Nat}   {\mathbb{N}}            \newcommand{\Ent}   {\mathbb{Z}}
    \newcommand{\set}[1]{\left\{#1\right\}}     
    
    \newcommand{\cat}[1]{\mathscr{#1}}              \renewcommand{\cal}[1]{\cat{#1}}

    \newcommand{\kk}    {\ensuremath{\mathbf{k}}}
                      
    \newcommand{\To}    {\longrightarrow}
  \renewcommand{\mod}   {\operatorname{mod}}
    \newcommand{\Hom}   {\operatorname{Hom}}    \newcommand{\End}   {\operatorname{End}}
        \newcommand{\rad}   {\operatorname{rad}}
    \newcommand{\ind}   {\operatorname{ind}}    \newcommand{\add}   {\operatorname{add}}

      \newcommand{\soc}   {\operatorname{soc}}
    \newcommand{\Ext}   {\operatorname{Ext}}    
    \newcommand{\D}     {\operatorname{D}}      
                    \newcommand{\Ann}   {\operatorname{Ann}}
    \newcommand{\projd} {{\rm pd}\,}            
    \newcommand{\gldim} {\operatorname{gl.dim.}}\newcommand{\repdim}{\operatorname{rep.dim.}}
  \renewcommand{\le}        {\leqslant}              \renewcommand{\ge}     {\geqslant}
    \newcommand{\Clu}   {\mathcal{C}}              \newcommand{\Tensor} {\operatorname{T}}

    \renewcommand{\vec}[1]   {\protect\vv{#1}}   
    \newcommand{\ivec}[1]   {\text{\reflectbox{\ensuremath{\protect\vv{\text{\reflectbox{$#1$}}}}}}}   

    \newcommand{\lra}{\To}

\newcommand{\bdim}{\operatorname{\mathbf{dim}}}
\newcommand{\op}{\operatorname{op}}
\newcommand{\cC}{\mathcal{C}}
\newcommand{\cP}{\mathcal{P}}
\newcommand{\cR}{\mathcal{R}}
\newcommand{\cX}{\mathcal{X}}

    \makeatletter
        \def\theequation{\thesection.\@arabic\c@equation}
        
        \def\dotsd{\,\mbox{}\mathinner{\mkern1mu\raise17\p@
            \vbox{\kern13\p@\hbox{.}}\mkern2mu
            \raise14\p@\hbox{.}\mkern2mu\raise11\p@\hbox{.}\mkern1mu}\mbox{}\,}
        \def\dotst{\,\mathinner{\mkern1mu\raise11\p@
            \vbox{\kern13\p@\hbox{.}}\mkern2mu
            \raise14\p@\hbox{.}\mkern2mu\raise17\p@\hbox{.}\mkern1mu}\,}
    \makeatother

\begin{document}

\title{The Representation Dimension of a Selfinjective Algebra of Wild Tilted Type}
\author{Ibrahim Assem}
    \address{D\'epartement de math\'ematiques, Facult\'e des sciences, Universit\'e de Sherbrooke,
            Sherbrooke, Qu\'ebec J1K 2R1, Canada.}
\author{Andrzej Skowro\'nski}
    \address{Faculty of Mathematics and Computer Science, Nicolaus Copernicus University,
           Chopina 12/18, 87-100 Toru\'n, Poland.}
\author{Sonia Trepode}
    \address{Departamento de Matem\'atica, Facultad de Ciencias Exactas y Naturales, Funes 3350,
            Universidad Nacional de Mar del Plata. Conicet, UNMDP, 7600 Mar del Plata, Argentina.}

\author{}


\begin{abstract}
    We prove that the representation dimension of a selfinjective algebra of wild tilted
    type is equal to three, and give an explicit construction of an Auslander generator of its module category.
    We also show that if a connected selfinjective algebra admits an acyclic generalised standard Auslander-Reiten component then its representation dimension is equal to three.
\end{abstract}

\maketitle

\section{Introduction}

Our objective in this paper is to explore the relation between the representation theory of an algebra, or more precisely the shape of its Auslander-Reiten components, and its homological invariants. We are in particular interested here in the representation dimension of an algebra, introduced by Auslander in \cite{A}, which measures in some way the complexity of the morphisms of the module category. There were several attempts to understand, or compute, this invariant, see, for instance, \cite{A}, \cite{EHIS}, \cite{ACW}. Special attention was given to algebras of representation dimension three. The reason for this interest is two-fold. Firstly, it is related to the finitistic dimension conjecture: Igusa and Todorov have proved that algebras of representation dimension three have a finite finitistic dimension \cite{IT}. Secondly, because Auslander's expectation was that the representation dimension would measure how far an algebra is from being representation-finite, there is a standing conjecture that the representation dimension of a tame algebra is at most three. Indeed, while there exist algebras of arbitrary \cite{Ro}, but finite \cite{I}, representation dimension, most of the best understood classes of algebras have representation dimension three. This is the case, for instance, for algebras obtained by means of tilting, such as tilted algebras \cite{APT}, iterated tilted algebras \cite{CHU} and quasitilted algebras \cite{O}. This is also the case for classes of selfinjective algebras related to the ones obtained via tilting, such as trivial extensions of iterated tilted algebras \cite{CP} and selfinjective algebras of euclidean type \cite{AST}. In both of these cases, the algebra considered is the orbit algebra of the repetitive algebra of some tilted algebra under the action of an infinite cyclic group of automorphisms.

It was then natural to consider next the class of selfinjective algebras of wild tilted type, introduced and studied in \cite{EKS}. A selfinjective algebra $A$ is called of \emph{wild tilted type} if $A$ is the orbit category of the repetitive category $\hat{B}$, in the sense of \cite{HW}, of a tilted algebra $B$ of wild type, under the action of an infinite cyclic group of automorphisms. Our first main theorem may now be stated.

\vskip 0.3cm

\noindent
{\bf Theorem A.}  \emph{Let $A$ be a connected selfinjective algebra of wild tilted type. Then $\repdim A = 3$.}

\vskip 0.3cm

Because the definition of our class is similar to that of the one considered in \cite{AST}, we are able to follow the same general strategy of proof as in that paper. In particular, our proof is constructive and we are able to explicitly describe an Auslander generator of the module category of $A$. However, because we are dealing with wild algebras, the necessary constructions are different.

Returning to our basic problem of relating the shape of Auslander-Reiten components to the representation dimension, we are led to consider the case of selfinjective algebras having an acyclic generalised standard component. We recall than an Auslander-Reiten component $\Gamma$ is called  \emph{generalised standard} \cite{S} whenever, for two modules $X,Y$ in $\Gamma$, we have $\rad_A^{\infty}(X,Y) = 0$, so that morphisms can be computed locally in that component. As an easy consequence of Theorem A and the results of \cite{AST}, \cite{SY}, we obtain the following nice and unexpected result.

\vskip 0.3cm

\noindent
{\bf Theorem B.} \emph{Let $A$ be a connected selfinjective algebra admitting an acyclic generalised standard Auslander-Reiten component. Then $\repdim A = 3.$}

\vskip 0.3cm

Thus, in this case, the good behaviour of an Auslander-Reiten component in the module category suffices to determine the representation dimension.

We now describe the contents of the paper. After an introductory section 2 in which we briefly fix the notation and recall useful facts, our section 3 is devoted to wild quasiserial algebras which will play the role of our building blocks  and section 4 is devoted to the gluings of such algebras. After recalling in section 5 necessary facts on repetitive categories of wild tilted algebras, we prove in section 6 our Theorem A using Galois coverings and we deduce Theorem B. Finally, the last section 7 is devoted to examples.

\section{Preliminaries}

\subsection{Notation}

Let $\kk$ be an algebraically closed field. By algebra $A$, we mean a basic, connected, associative finite dimensional $\kk$-algebra with an identity. It is well-known that there exists a connected bound quiver $Q_A$ and an admissible ideal $I$ of the path algebra $\kk Q_A$ such that $A \cong \kk Q_A/ I$,  see, for instance \cite{ASS}.
Equivalently, $A$ may be considered as a $\kk$-category whose objects are the points of $Q_A$,
and the set of morphisms $A(x,y)$ from $x$ to $y$ is the quotient of the $\kk$-vector space $\kk Q_A(x,y)$  having as basis all paths from $x$ to $y$ by the subspace $I(x,y) = I \cap \kk Q_A (x,y)$, see~\cite{BG}.
A full subcategory $C$ of $A$ is \emph{convex} if,
    for each path $\xymatrix@1@C=12pt{x_0 \ar[r]& x_1\ar[r]& \cdots \ar[r]& x_n}$
    in $A$ with $x_0, x_n$ in $C$,
    we have $x_i$ in $C$ for each $i$.
The algebra $A$ is called \emph{triangular} if $Q_A$ is acyclic.

By $A$-module,  we mean a finitely generated right $A$-module.
We denote by $\mod A$ their category and
   by $\ind A$ a full subcategory consisting of a complete set of representatives of
   the isoclasses (isomorphism classes) of indecomposable $A$-modules.
For a point $x$ in  $Q_A$, we denote by
 $P(x)$, or $I(x)$, or $S(x)$
 the corresponding indecomposable projective, or indecomposable injective, or simple $A$-module.
 We denote by $\projd M$  the projective dimension of a module $M$. The global dimension of $A$ is then denoted by  $\gldim A$.
 Given a module $M$, the additive full subcategory of all direct summands of finite
 direct sums of $M$ is denoted by  $\add M$.
 Given two full subcategories $\cal{C}, \cal{D}$ of $\ind A$, such that $\Hom_A(M,N) = 0$ for all $M$ in $\cal{C}$, $N$ in $\cal{D}$,  the notation  $\cal D \vee \cal C$ represents the full subcategory of $\ind A$ having as object class the union of the object classes of $\cal{C}$ and $\cal{D}$.
Finally, $\D = \Hom_{\kk}(-,\kk)$ denotes the usual duality between $\mod A$ and $\mod A^{\rm{op}}$.

A \emph{path} in $\ind A$ from $M$ to $N$ is a sequence of nonzero morphisms
   \[\xymatrix@C=18pt{M = M_0 \ar[r]& M_1 \ar[r]&   \cdots  \ar[r]& M_t = N} \tag{*}\]
with all $M_i$ indecomposable.
We then say that $M$ is a \emph{predecessor} of $N$, or that $N$ is a \emph{successor} of $M$. We define similarly predecessor and successor of a class $\cal{C}$ of $\mod A$ (for instance, of an Auslander-Reiten component): we say that $M$ is a predecessor (or successor) of $\cal{C}$ if there exists $N$ in $\cal{C}$ such that $M$ precedes $N$ (or succedes it, respectively).

We use freely properties of the Auslander-Reiten quiver $\Gamma(\mod A)$ of $A$ and tilting theory, for which we refer to \cite{ASS}, \cite{ARS}, \cite{SS1}, \cite{SS2}.
Points in $\Gamma(\mod A)$ are identified with the corresponding indecomposable $A$-modules. Similarly, (parts of) components of $\Gamma(\mod A)$ are identified with the corresponding full subcategories of $\ind A$.

\subsection{Representation dimension}

    Let $A$ be a non-semisimple algebra. We recall from \cite{A} that its \emph{representation dimension} $\repdim A$ is the infimum of
    the global dimensions of the algebras $\End M$, where the module $M$ is at the same time a generator
    and a cogenerator of $\mod A$. Clearly, if $M$ is a generator and a cogenerator of $\mod A$,
then it can be written as $M=A\oplus \D A\oplus M'$, for some $A$-module $M'$.
If $M$ is a generator and cogenerator of $\mod A$  and moreover $\repdim A =\gldim\End M$,
   then $M$ is called an \emph{Auslander generator} for $\mod A$. Thus, all indecomposable projective and all indecomposable injective $A$-modules are direct summands of any Auslander generator of $\mod A$.

We now state a criterion for an algebra to have representation dimension at most three.
Let $M$ be any $A$-module. Given an $A$-module $X$,
a morphism $f_0: M_0 \lra X$  with $M_0 \in  \operatorname{add }M$ is an \emph{$\add M$-approximation} if,
   for any morphism $f_1: M_1 \to X$ with $M_1 \in \add M$, there exists $g: M_1 \to M_0$ such that $f_1 = f_0g$:
   \[\xymatrix@R=6pt{
                     M_1\ar@{.>}[dd]_g
                        \ar[drr]^{f_1}  &&    \\
                                       &&   X~.\\
                     M_0\ar[urr]^{f_0} &&   }
   \] An $\add M$-approximation is \emph{(right) minimal} if each morphism $g: M_0 \to M_0$ such that $f_0g= f_0$ is an isomorphism. A short exact sequence
    \[\xymatrix{ 0\ar[r]    &   M_1\ar[r]   &   M_0\ar[r]^{f_0} &   X\ar[r] &   0   }\]
with $M_1$, $M_0\in\add M$ is an \emph{$\add M$-approximating sequence} if $f_0:M_0\to X$ is an $\add M$-approximation of $X$. If moreover $f_0:M_0\to X$ is minimal, then it is called a \emph{minimal $\add M$-approximating sequence}. It is proved in \cite{ACW}(1.7), \cite{AST} (2.3) that if $X$ admits an $\add M$-approximating sequence, then it admits a minimal $\add M$-approximating sequence, which is moreover a direct summand of any $\add M$-approximating sequence. The following lemma which combines \cite{AST}(2.2) and (2.4) is central to our considerations.

\begin{lemma}
\begin{enumerate}
\item[(a)] A short exact sequence
 \[\xymatrix{ 0\ar[r]   &   M_1\ar[r]   &   M_0\ar[r]^{f_0} &   X\ar[r] &   0   }\] with $M_0, M_1 \in \add M$ is a minimal $\add M$-approximating sequence if and only if the induced sequence of functors
 \[\xymatrix{0\ar[r]& \Hom_A(-,M_1)|_{\add M} \ar[r]& \Hom_A(-,M_0)|_{\add M} \ar[r]& \Hom_A(-,X)|_{\add M} \ar[r]& 0}\]
 is exact.
 \item [(b)] $\gldim \End M \leq 3$ if and only if each $A$-module admits a (minimal) $\add M$-approximating sequence.
     Moreover, in this case, $\repdim A \leq 3$.
\end{enumerate}
\end{lemma}

\subsection{Tilted Algebras}

Let $A$ be an algebra.
An $A$-module $T$ is a \emph{tilting module} if $\projd_A T \leq 1$,  $\Ext^1_A(T,T)= 0$
and the number of isoclasses of indecomposable summands of $T$ equals the rank of the Grothendieck group $K_0 (A)$ of $A$.
Let $Q$ be a finite, connected and acyclic quiver and $\kk Q$ its path algebra. An algebra is \emph{tilted of type $Q$} if there exists a tilting $\kk
Q$-module $T$ such that $A = \End_{\kk Q} T$.

We now state a well-known criterion for an algebra to be tilted. We recall that an $A$-module $M$ is \emph{sincere} if $\Hom_A(P,M) \not= 0$ for any projective $A$-module $P$. It is \emph{convex} if for any path

\[\xymatrix@C=18pt{M_0 \ar[r] & M_1 \ar[r] &    \cdots  \ar[r] & M_t} \tag{*}\] in $\ind A$ with $M_0,M_1 \in \add M$ we have $M_i \in \add M$ for all $i$. Finally, $M$ is a \emph{slice module} if it is sincere, convex, and for any almost split sequence $\xymatrix@1@C=12pt{0 \ar[r]& X \ar[r]& Y \ar[r]& Z \ar[r]& 0}$ in $\mod A$, at most one of $X,Z$ lies in $\add M$, and  moreover, if an indecomposable summand of $Y$ lies in $\add M$ then exactly one of $X,Z$ lies in $\add M$. The full subquiver of $\Gamma(\mod A)$ generated by the summands of a slice module is a \emph{complete slice}, see \cite{Ri2}. It is then shown, see, for instance \cite{Ri}, that $A$ is tilted if and only if it admits a slice module (or, equivalently, a complete slice).

The representation dimension of any tilted algebra is at most three. Actually, we have the following facts, proved in \cite{APT}, see also \cite{AST}(2.5).

\begin{lemma}
    Let $A$ be a tilted algebra, and $T$ a slice module.
    Then
    \begin{enumerate}[\indent(a)]
        \item If $X$ is an $A$-module generated by $T$, then there exists a minimal
            $\add(T \oplus \D A)$-approximating sequence for $X$ of the form
                \[\xymatrix@1{0 \ar[r]& T_1 \ar[r]& T_0 \oplus I_0 \ar[r]& X \ar[r]& 0}\]
            with $T_1$, $T_0 \in \add T$ and $I_0$ injective.
        \item The module $M = A \oplus \D A \oplus T$ is an Auslander generator for $\mod A$ and $\repdim A\le3$.
    \end{enumerate}
\end{lemma}

\subsection{Tilted algebras of wild type}

We need to recall a few results from \cite{EKS} on tilted algebras of wild type. A tilted algebra of type $Q$ is
\emph{concealed} of type $Q$ if it is the endomorphism algebra of a tilting $\kk Q$-module all of whose summands belong to the postprojective (equivalently, preinjective) component of $\Gamma(\mod \kk Q)$.

A \emph{(truncated) branch} in a point $a$ is any finite, connected, full bound subquiver, containing $a$, of the following infinite tree bound by all possible relations of the form $\alpha\beta = 0$, see \cite{Ri},
%
%
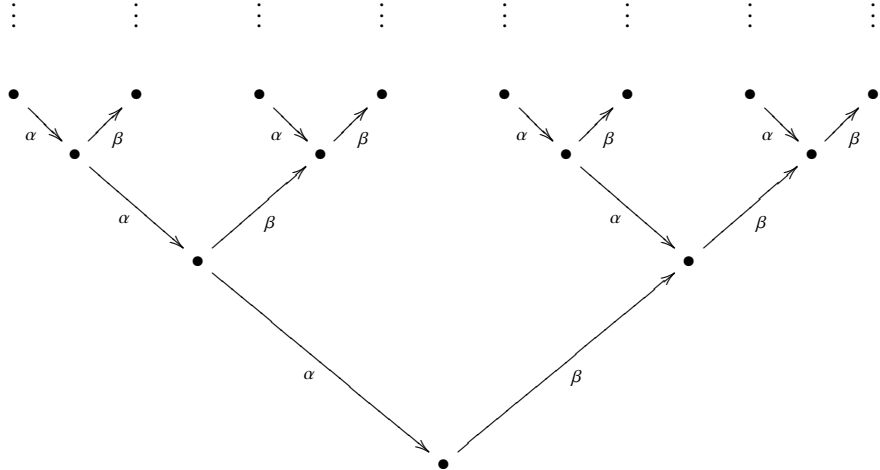
$$
 \xymatrix@C=1pc@R=1pc{
   \vdots && \vdots && \vdots  && \vdots && \vdots && \vdots && \vdots  && \vdots \\
   \bullet \ar[rd]_{\alpha} && \bullet && \bullet \ar[rd]_{\alpha} && \bullet &&
     \bullet \ar[rd]_{\alpha} && \bullet && \bullet \ar[rd]_{\alpha} && \bullet\\
   & \bullet \ar[ru]_{\beta} \ar[rrdd]_{\alpha} &&&& \bullet  \ar[ru]_{\beta} & &&
     & \bullet \ar[ru]_{\beta} \ar[rrdd]_{\alpha} &&&& \bullet  \ar[ru]_{\beta} &\\
   \\
   &&& \bullet \ar[rruu]_{\beta} \ar[rrrrdddd]_{\alpha} &&& && &&& \bullet \ar[rruu]_{\beta} && \\
   \\
   \\ \\
   &&& &&&& \bullet \ar[rrrruuuu]_{\beta}\\
 }
$$
Let $A = {\kk Q}_A/I$ be a bound quiver algebra and $A' = {\kk Q'}/I'$ a full convex subcategory of $A$ having $a$ as source (or target). Then $A$ is obtained from $A'$ by \emph{rooting} a branch ${\kk Q''}/I''$ in $a$ provided  $(Q'',I'')$ is a full bound subquiver of $(Q,I)$ such that $Q'_0 \cap Q''_0 = \{a\}$, $Q_0 = Q'_0 \cup Q''_0$ and $I$ is generated by $I'$ and $I''$.

We define ray and coray modules.  Let $\cal{C}$ be an Auslander-Reiten component. An indecomposable module $X$ in $\cal{C}$ is a \emph{ray module} if there exists an infinite sectional path
%
%
 $$ X= X_0 \longrightarrow X_1 \longrightarrow \cdots \longrightarrow X_t \longrightarrow \cdots$$ such that, for any $i \geq 0$, the subpath  $X = X_0 \longrightarrow X_1 \longrightarrow     \cdots  \longrightarrow X_i$ is the unique sectional path of length $i$ with source in $X$. \emph{Coray modules} are defined dually.

 Let now $C$ be a wild concealed algebra and $K_1, \cdots K_m$ be branches. An algebra $B$ is a \emph{branch extension} of $C$ if there exists a sequence $B_0 = C, B_1, \cdots B_m = B$ of algebras such that, for each $j$, the algebra $B_{j+1}$ is obtained from the one-point extension $B_j[E_{j+1}]$, where $E_{j+1}$ is a ray $B_j$-module, by rooting the branch $K_{j+1}$ at the extension point. \emph{Branch coextensions} of wild concealed algebras are defined dually.

The following result is \cite{EKS}(2.8).

\begin{proposition}
Let $Q$ be a connected wild quiver, $T$ a tilting $\kk Q$-module and $B = \End_{\kk Q} T$.
\begin{enumerate}[\indent(a)]
\item If $T$ has no preinjective direct summand, then there exists a  wild concealed quotient algebra $C$ of $B$ whose postprojective component is the unique postprojective component of $\Gamma(\mod B)$. Moreover, $B$ is a branch extension of $C$.
 \item   If $T$ has no postprojective direct summand, then there exists a  wild concealed quotient algebra $C$ of $B$ whose preinjective component is the unique preinjective component of $\Gamma(\mod B)$. Moreover, $B$ is a branch coextension of $C$.
\end{enumerate}
\end{proposition}

 As an illustration, and because this picture will become useful in the sequel, we show the shape of the Auslander-Reiten quiver of a tilted algebra $B$ which is of type $(a)$ above (that is, the endomorphism algebra of a tilting module without preinjective direct summands)
%
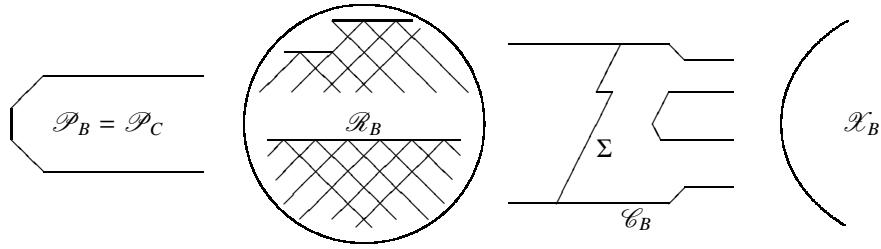
\[
\xy
0;/r.5pc/:
,
(-10,-5)="2a" ;
(-8,-11)="2b" ;
(10,-5)="4a" ;
(8,-11)="4b" ;
(0,0)
,{\ellipse(7.5,7.5){-}},*{\cR_B}
;p+(-10,-3)
;p+(-10,0)*\dir{}="b",**\dir{-},"b"
;p+(-2,2)*\dir{}="b",**\dir{-},"b"
;p+(0,1)*\dir{}="b",**\dir{-},"b"
;p+(6,0)*{\cP_B = \cP_C};p+(-6,0)
;p+(0,1)*\dir{}="b",**\dir{-},"b"
;p+(2,2)*\dir{}="b",**\dir{-},"b"
;p+(10,0)*\dir{}="b",**\dir{-},"b"
;p+(10,-3)
;p+(23,4)
;p+(-3,0)*\dir{}="b",**\dir{-},"b"
;p+(-1,1)*\dir{}="b",**\dir{-},"b"
;p+(-10,0)*\dir{}="b",**\dir{-},"b"
;p+(7,0)
;p+(-1.5,-3)*\dir{}="b",**\dir{-},"b"
;p+(1,0)*\dir{}="b",**\dir{-},"b"
;p+(-3.5,-7)*\dir{}="b",**\dir{-},"b"
;p+(-3,0)
;p+(6,3.5)*{\Sigma};p+(-6,-3.5)
;p+(8,-1)*{\Clu_B};p+(-8,1)
;p+(10,0)*\dir{}="b",**\dir{-},"b"
;p+(1,1)*\dir{}="b",**\dir{-},"b"
;p+(3,0)*\dir{}="b",**\dir{-},"b"
 ;p+(0,3)
 ;p+(-4.5,0)*\dir{}="b",**\dir{-},"b"
 ;p+(-.5,1)*\dir{}="b",**\dir{-},"b"
 ;p+(1,2)*\dir{}="b",**\dir{-},"b"
 ;p+(4,0)*\dir{}="b",**\dir{-},"b"
 ;p+(0,-6)
;p+(13,4)
 ,{\ellipse(10,8):a(210),,=:a(120){-}}
 ;p+(-5,0)*{\cX_B};p+(5,0)
;p+(-13,-4)
;p+(-23,4)
;p+(-6,-1)
 ;p+(12,0)*\dir{}="b",**\dir{-},"b"
 ;p+(-11,0)
 ;p+(-1,-1)
 ;p+(1,1)*\dir{}="b",**\dir{-},"b"
 ;p+(5.5,-5.5)*\dir{}="b",**\dir{-},"b"
 ;p+(-5.5,5.5)
 ;p+(-.5,-2.5)
 ;p+(2.5,2.5)*\dir{}="b",**\dir{-},"b"
 ;p+(5,-5)*\dir{}="b",**\dir{-},"b"
 ;p+(-7,5)
 ;p+(0,-4)
 ;p+(4,4)*\dir{}="b",**\dir{-},"b"
 ;p+(5,-5)*\dir{}="b",**\dir{-},"b"
 ;p+(-9,5)
 ;p+(1,-5)
 ;p+(5,5)*\dir{}="b",**\dir{-},"b"
 ;p+(4,-4)*\dir{}="b",**\dir{-},"b"
 ;p+(-10,4)
 ;p+(3,-5)
 ;p+(5,5)*\dir{}="b",**\dir{-},"b"
 ;p+(2.5,-2.5)*\dir{}="b",**\dir{-},"b"
 ;p+(-10.5,2.5)
 ;p+(4.5,-5.5)
 ;p+(5.5,5.5)*\dir{}="b",**\dir{-},"b"
 ;p+(1,-1)*\dir{}="b",**\dir{-},"b"
 ;p+(-11,1)
 ;p+(0,5.5)
 ;p+(3,0)*\dir{}="b",**\dir{-},"b"
 ;p+(0,2)
 ;p+(5,0)*\dir{}="b",**\dir{-},"b"
 ;p+(-5,0)
 ;p+(4.5,-4.5)*\dir{}="b",**\dir{-},"b"
 ;p+(-4.5,4.5)
 ;p+(0,-2)
 ;p+(2.5,-2.5)*\dir{}="b",**\dir{-},"b"
 ;p+(-2.5,4.5)
 ;p+(-4.5,-4.5)
 ;p+(2.5,2.5)*\dir{}="b",**\dir{-},"b"
 ;p+(2.5,-2.5)*\dir{}="b",**\dir{-},"b"
 ;p+(-.5,4.5)
 ;p+(-2.5,-4.5)
 ;p+(4.5,4.5)*\dir{}="b",**\dir{-},"b"
 ;p+(4.5,-4.5)*\dir{}="b",**\dir{-},"b"
 ;p+(-6.5,4.5)
 ;p+(-.5,-4.5)
 ;p+(4.5,4.5)*\dir{}="b",**\dir{-},"b"
 ;p+(4.5,-4.5)*\dir{}="b",**\dir{-},"b"
 ;p+(-8.5,4.5)
 ;p+(1.5,-4.5)
 ;p+(4.,4.)*\dir{}="b",**\dir{-},"b"
 ;p+(-5.5,.5)
\endxy
\]
Here, $\cal{P}_B$ is the unique postprojective component of $\Gamma(\mod B)$, it coincides with the postprojective component $\cal{P}_C$ of $\Gamma(\mod C)$ where $C$ is wild concealed and such that $B$ is a branch extension of $C$.
 The family of components $\cal{R}_B$ consists in components of the form $\Ent\mathbb{A}_{\infty}$ or of such components with one or several rays inserted. The connecting component $\cal{C}_B$ contains no projective, but may contain injectives. It also contains a complete slice (which is thus left stable). Finally, $\cal{X}_B$ represents the remaining part of the module category.

 Clearly,  the tilted algebras of type $(b)$ are dual to the ones of type $(a)$.

 We refer to \cite{K1}, \cite{K2}, \cite{K3}, \cite{KS}, \cite{Ri2} for more details on representation theory of tilted algebras of wild type.

 \subsection{Reflections}

 Let $B$ be a triangular algebra and $i$ a sink in its quiver $Q_B$. We define $T^+_i B$ to be the one-point extension of $B$ by the indecomposable injective $B$-module $I(i)$:
 $$T^+_i B = B[I(i)] = \left(  \begin{array}{cc}
                            B & 0 \\
                            I(i) & k \\
                    \end{array} \right)$$ where addition is the usual addition of matrices, while multiplication is induced from the right module structure of $I(i)$. Thus $Q_{T^+_iB}$ contains $Q_B$ as a full convex subquiver and has exactly one additional point which is a source, denoted as $i'$. The \emph{reflection} $S_i^+B$ of $B$ at $i$ is the full subcategory of $T^+_iB$ given by the objects of $B$ except the sink $i$. Thus, the sink $i$ in $Q_B$ is replaced by the source $i'$ in the quiver $\sigma^{+}_{i} Q_B = Q_{S^+_i B}$.
 We may iterate this procedure: a \emph{reflection sequence of sinks} is a sequence $i_1,\dotsc,i_t$ of points in $Q_B$ such that, for each $s\in\set{1,\dotsc,t}$, the point $i_s$
is a sink in the quiver $\sigma_{i_{s-1}}^+\cdots \sigma_{i_1}^+Q_B = Q_{S^+_{i_{s-1}}\cdots S^+_{i_1}B}$. Then, the reflection $B' = S^+_{i_s}(S^+_{i_{s-1}}\cdots S^+_{i_1}B)$ is defined.

\section{Wild quasiserial algebras}

\subsection{The definition}

Let $Q$ be a connected wild quiver, $T^-$ a tilting $\kk Q$-module without postprojective direct summand and $B^- = \End T^-$. Because of Proposition 2.4, $B^-$ is a branch coextension of some wild concealed algebra $C$.

According to \cite{EKS}, Theorem 3.5 (see, also, the algorithmic procedure in \cite{ANS}(2.5)) there exists a reflection sequence of sinks $i_1,\dotsc,i_r$ with all $i_k$ belonging to the branches (not to $C$) and having the property that
$B^+ = S^+_{i_r}\cdots S^+_{i_1}B^-$ is the endomorphism algebra of a tilting $\kk Q$-module without preinjective  direct summand. Thus, again, because of Proposition 2.4, $B^+$ is a branch extension of $C$.

\begin{definition}
Let $B^-$ and $i_1,\dotsc,i_r$ be as above. Then the iterated one-point extension $B = T^+_{i_r}\cdots T^+_{i_1}B^-$ is a \emph{wild quasiserial} algebra.
\end{definition}

\begin{remark}

Using the description of the Auslander-Reiten quivers of $B^-$ and $B^+$ as given in 2.4, it is not hard to see that the Auslander-Reiten quiver of $B$ has the following shape
%
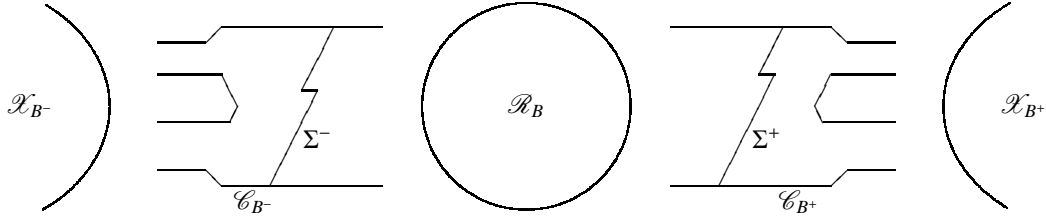
\[
\xy
0;/r.5pc/:
,
(-10,-5)="2a" ;
(-8,-11)="2b" ;
(10,-5)="4a" ;
(8,-11)="4b" ;
(0,0)
,{\ellipse(6.5,6.5){-}},*{\cR_B}
;p+(-23,4)
;p+(3,0)*\dir{}="b",**\dir{-},"b"
;p+(1,1)*\dir{}="b",**\dir{-},"b"
;p+(10,0)*\dir{}="b",**\dir{-},"b"
;p+(-3,0)
;p+(-2,-4)*\dir{}="b",**\dir{-},"b"
;p+(1,0)*\dir{}="b",**\dir{-},"b"
;p+(-3,-6)*\dir{}="b",**\dir{-},"b"
;p+(7,0)
;p+(-4,3)*{\Sigma^-};p+(4,-3)
;p+(-8,-1)*{\Clu_{B^-}};p+(8,1)
;p+(-10,0)*\dir{}="b",**\dir{-},"b"
;p+(-1,1)*\dir{}="b",**\dir{-},"b"
;p+(-3,0)*\dir{}="b",**\dir{-},"b"
 ;p+(0,3)
 ;p+(4.5,0)*\dir{}="b",**\dir{-},"b"
 ;p+(.5,1)*\dir{}="b",**\dir{-},"b"
 ;p+(-1,2)*\dir{}="b",**\dir{-},"b"
 ;p+(-4,0)*\dir{}="b",**\dir{-},"b"
 ;p+(0,-6)
;p+(-13,4)
 ,{\ellipse(10,8):a(210),,=:a(120){-}}
 ;p+(5,0)*{\cX_{B^-}};p+(-5,0)
;p+(13,-4)
;p+(46,8)
;p+(-3,0)*\dir{}="b",**\dir{-},"b"
;p+(-1,1)*\dir{}="b",**\dir{-},"b"
;p+(-10,0)*\dir{}="b",**\dir{-},"b"
;p+(7,0)
;p+(-1.5,-3)*\dir{}="b",**\dir{-},"b"
;p+(1,0)*\dir{}="b",**\dir{-},"b"
;p+(-3.5,-7)*\dir{}="b",**\dir{-},"b"
;p+(-3,0)
;p+(6,3)*{\Sigma^+};p+(-6,-3)
;p+(8,-1)*{\Clu_{B^+}};p+(-8,1)
;p+(10,0)*\dir{}="b",**\dir{-},"b"
;p+(1,1)*\dir{}="b",**\dir{-},"b"
;p+(3,0)*\dir{}="b",**\dir{-},"b"
 ;p+(0,3)
 ;p+(-4.5,0)*\dir{}="b",**\dir{-},"b"
 ;p+(-.5,1)*\dir{}="b",**\dir{-},"b"
 ;p+(1,2)*\dir{}="b",**\dir{-},"b"
 ;p+(4,0)*\dir{}="b",**\dir{-},"b"
 ;p+(0,-6)
;p+(13,4)
 ,{\ellipse(10,8):a(210),,=:a(120){-}}
 ;p+(-5,0)*{\cX_{B^+}};p+(5,0)
;p+(-13,-4)
\endxy
\]
Here $\cal{X}_{B^-}$ consists of all Auslander-Reiten components of $\Gamma(\mod B^-)$ which precede the connecting component $\cal{C}_{B^-}$ of $\Gamma(\mod B^-)$. The family $\cal{R}_B$ consists of components whose stable parts are of the form $\Ent \mathbb{A}_{\infty}$ and in which all projectives, and all injectives, are projective-injective. Similarly, $\cal{C}_{B^+}$ is the connecting component of $\Gamma(\mod B^+)$ and $\cal{X}_{B^+}$ consists of all Auslander-Reiten components which follow $\cal{C}_{B^+}$.

Observe that $\cal{C}_{B^-}$ contains a complete slice $\Sigma^-$ of $\mod B^-$. Clearly, $\Sigma^-$  is not a complete slice in $\mod B$ in general, because $\cal{R}_B$ may contain projective-injectives. However, $\Sigma^-$ is a right section in the sense of \cite{A1}. Dually, a  complete slice $\Sigma^+$ of $\mod B^+$ lies in $\cal{C}_{B^+}$, it is not in general a complete slice in $\mod B$, but it is a left section. Moreover, one has $B^- = B /\Ann \Sigma^-$ and $B^+ = B/ \Ann \Sigma^+$. Also notice that it follows from the description above that $B^-$ is the left support algebra of $B$, and $B^+$ its right support algebra, in the sense of \cite{ACT}. For a detailed example, we refer the reader to section 7.
\end{remark}
\subsection{The representation dimension}
Our first objective is to prove that wild quasiserial algebras have representation dimension $3$. Let $B$ be a wild quasiserial algebra. Using the notation introduced in 3.1, we let $\Sigma^-$ be a complete slice in $\mod B^-$ and
$\Sigma^+$ a complete slice in $\mod B^+$. We may, without loss of generality, choose $\Sigma^+$ in such a way that every indecomposable $B$-module with support lying completely in the extension branches of $B^+$ is a successor of $\Sigma^+$. This is possible, because there are only finitely many isoclasses of such indecomposable modules. As a consequence, the restriction to $B^-$ of any predecessor of $\Sigma^+$ in the component $\cal{C}_B^+$ is nonzero. We let $T^-$ and $T^+$ denote the respective slice modules of $\Sigma^-$ and $\Sigma^+$ and set
$$M = B \oplus \D B \oplus T^{-} \oplus T^+ \oplus \D B^{-}.$$
We prove below that $M$ is an Auslander generator of $\mod B$. Also useful is the module
$$N = B^- \oplus T^{-} \oplus \D B^{-}.$$
Indeed recall that, because of Lemma 2.3, every indecomposable $B^-$-module admits a minimal $\add N$-approximating sequence. Our main technical tool is the following proposition which has almost the same statement as \cite{AST}, Proposition 3.3.

\begin{proposition}
    Let $B$ be a wild quasiserial algebra and $X$  an indecomposable $B$-module
    whose restriction $Y$ to $B^{-}$ is nonzero.
    Let also
        \[\xymatrix@R=12pt{
           0 \ar[r]& L \ar[r]^-r    & N_0  \ar[r]^-q & Y \ar[r]& 0\\
           0 \ar[r]& L' \ar[r]^{i'} & P \ar[r]^{p'}& X/Y \ar[r]&0 }\]
    be respectively a minimal $\add N$-approximating sequence for $Y$ and a
    projective cover of $X/Y$ in $\mod B$.
    Then there exists a $B$-module $K$ such that we have exact sequences
      \[\xymatrix@R=12pt{
         0 \ar[r]& K \ar[r]^-s & N_0 \oplus P    \ar[r]^-t & X \ar[r]& 0\\
         0 \ar[r]& L \ar[r]^\ell & K \ar[r]^{\ell'}& L' \ar[r]&0. }\]
   Moreover, $K\cong L\oplus L'$.
   In particular, $K\in\add N$.
\end{proposition}
\begin{proof}
The proof of \cite{AST}, Proposition 3.3, applies here verbatim.
\end{proof}

\subsection{The theorem}
\begin{thm}
    Let $B$ be a wild quasiserial algebra. Then $\repdim B = 3$.
\end{thm}
\begin{proof}

   Let $X$ be an indecomposable $B$-module. Using again the notation introduced in 3.1, we assume first that $X$ is cogenerated by $\tau T^-$ (that is, $X$ is a proper predeccessor of $\Sigma^-$). In this case, we have $\projd X \leq 1$ and there is nothing to prove. Otherwise, assume $X$ to be a module of $\cal{C}_{B^-}$ which is generated by $T^-$ (that is, which is a successor of $\Sigma^-$). Then because of Lemma 2.2, there exists a minimal $\add(B^-\oplus T^-\oplus \D B^-)$-approximating sequence
    \[\xymatrix@C=40pt@R=20pt{ 0 \ar[r]& M_1 \ar[r]& M_0\ar[r] & X \ar[r]& 0. }\]
   Because  $B^-$ is closed under succesors, the projective $B^-$-modules are also projective $B$-modules, so $M_1$, $M_0\in\add M$.
   Let $f:M'\to X$ be a nonzero morphism with  $M' \in \add M.$ We may assume, without loss of generality, that $M'$ is an indecomposable summand of $M$. Because $X$ is a successor of $\Sigma^-$, we have $M'\in\add(B^-\oplus T^-\oplus \D B^-)$ hence $f$ factors through $M_0$, because of Lemma 2.2. Thus, we have an $\add M$-approximating sequence.

   Now, let $X$ belong to one of the components which are successors of $\cal{C}_{B^-}$  and be such that the restriction $Y$ of $X$ to $B^-$ is nonzero. In particular, this is the case for all modules in the family of components $\cal{R}_B$ and for the predecessors of $\Sigma^+$ in $\cal{C}_{B^+}$ (this is due to our choice of the slice $\Sigma^+$). Because of Proposition 3.2, there exists an exact sequence

           \[\xymatrix@1@R=12pt{
              0 \ar[r]& K \ar[r]^-s    & T_0\oplus I_0\oplus P  \ar[r]^-t & X \ar[r]& 0  }\]
    with $K$, $T_0\oplus I_0\oplus P\in\add(B^-\oplus T^- \oplus \D B^-)$.
    We show that it is an $\add M$-approximating sequence.
    Let $f:M'\to X$ be a nonzero morphism with $M'\in \add M$. Because $f$ is nonzero, we have $M'\notin\add(T^+\oplus \D B^+)$. Therefore, $M'\in\add(B \oplus T^- \oplus \D B^-)$.
    If $M'$ is projective then $f$ trivially lifts to a morphism $M'\to T_0\oplus I_0\oplus P$.
    If $M'\in\add(T^-\oplus\D B^-)$, then $f(M') \subset Y = \Hom_B(B^-,X)$. Applying Proposition 3.2, we see that $f$ lifts to a morphism $M'\to T_0\oplus I_0$ and consequently to a morphism $M'\to T_0\oplus I_0\oplus P$.

    Because indecomposable $B$-modules whose restriction to $B^-$ is zero have support lying completely in the extension branches of $B^+$, they are successors of $\Sigma^+$, due to the construction of the latter.
    Therefore there only remains to consider the case where $X$ is a successor of $\Sigma^+$.
    If this is the case, then, because of Lemma 2.2, there exists a minimal $\add(T^+\oplus \D B^+)$-approximating sequence
           \[\xymatrix@1@R=12pt{
              0 \ar[r]& T_1' \ar[r]   & T_1\oplus I_1 \ar[r] & X \ar[r]& 0 . }\]
    Because $B^+$-injectives are also $B$-injectives, we have $T_1'$, $T_1\oplus I_1 \in\add M$.
    Let $f:M'\to X$ be a nonzero morphism with $M'\in\add M$ indecomposable.
    If $M'\in\add(T^+\oplus \D B^+)$, then clearly $f$ lifts to a morphism $M'\to T_1\oplus I_1$.
    If $M'\notin \add(T^+\oplus \D B^+)$, then $f$ must factor through $\Sigma^+$ and thus also lifts to a
    morphism $M'\to T_1\oplus I_1$.
    This finishes the proof.
\end{proof}

\section{Tilted Gluings}
\subsection{The definition}

In \cite{AST} the notion of finite gluing of algebras was introduced. Given two algebras of representation dimension three, the gluing process allows to construct a larger algebra having the same representation dimension. We use here the same strategy, but have to modify the definitions of \cite{AST} to make them suit our needs.

\begin{definition} An algebra $B$ is called  \emph{right (or left) admissible} if its Auslander-Reiten quiver admits a separating, acyclic and generalised standard component $\cal{C}_{B^+}$ containing a left section $\Sigma^+$ (or $\cal{C}_{B^-}$ containing a right section $\Sigma^-$, respectively).
\end{definition}

Admissibility is a condition made up for tilted algebras. Indeed, let $B$ be the endomorphism algebra of a tilting module over a hereditary algebra without postprojective direct summands. Then the connecting component of $\Gamma(\mod B)$ is separating, acyclic, generalised standard and moreover, it contains a right section. Therefore $B$ is left admissible. Dually, if $B$ is the endomorphism algebra of a tilting module without preinjective direct summands, then $B$ is right admissible. Conversely, if $B$ is an algebra having a generalised standard component containing a left, or right, section $\Sigma$, then because of \cite{A1}(3.6), $B/\Ann \Sigma$ is a tilted algebra having $\Sigma$ as complete slice.

Notice also that any algebra having a preinjective (or postprojective) component is right (or left, respectively) admissible.

Before defining gluings, we recall a notation. Let $\Sigma$ be a right, or a left, section in an acyclic component of the Auslander-Reiten quiver of an algebra $A$. Then we denote by $\ivec{\Sigma}$ (or $\vec{\Sigma}$) the set of all indecomposable $A$-modules which are predecessors (or successors, respectively) of $\Sigma$ in $\ind A$.

\begin{definition}
    Let $B_1$  be a right admissible and $B_2$ a left admissible algebra. Then an algebra $A$ is said to be a \emph{tilted gluing} of $B_1$ and $B_2$, in symbols $A = B_1*B_2$, if
    \begin{itemize}
        \item [\indent(FG1)] $\Gamma(\mod A)$ has a separating component $\cat G$ such that:
            \begin{enumerate}[\indent(1)]
                \item $\cal G$ contains a left section isomorphic to $\Sigma^+_{B_1}$
                    and the indecomposable $A$-modules in $\cat G$ which precede it are exactly those
                    of $\ivec{\Sigma}^+_{B_1}\cap \cat C^+_{B_1}$,
                \item $\cal G$ contains a right section isomorphic to $\Sigma^-_{B_2}$
                    and the indecomposable $A$-modules in $\cat G$ which succede it are exactly those
                    of $\vec{\Sigma}^{-}_{B_2} \cap \cat C^-_{B_2}$,
                \item $(\ivec{\Sigma}^+_{B_1} \cap \cat C^+_{B_1}) \cup
                        (\vec{\Sigma}^-_{B_2} \cap \cat C^-_{B_2})$ is cofinite in $\cat G$;
            \end{enumerate}
        \item [\indent(FG2)] The remaining indecomposable $A$-modules belong to one of two classes:
            \begin{enumerate}[\indent(1)]
                \item those which precede $\cal G$ are the indecomposable $B_1$-modules
                        in $\ivec{\Sigma}^+_{B_1} \setminus \cat C^+_{B_1}$,
                \item those which succede $\cal G$ are the indecomposable $B_2$-modules
                        in $\vec{\Sigma}^-_{B_2} \setminus \cat C^-_{B_2}$.
            \end{enumerate}
    \end{itemize}
\end{definition}

Thus the separating component $\cat G$, which we call the \emph{glued component}, induces a decomposition

   \[\ind A = (\ivec{\Sigma}^+_{B_1} \setminus \cal{C}^+_{B_1})
                   \,\vee\; \cat G \;
                   \vee\,  (\vec{\Sigma}^-_{B_2} \setminus \cal{C}^-_{B_2}).\]

\[\xymatrix@C=10pt@R=10pt{
   \ar@{-}@(r,r)[dddd]
    &&                  &&                  &&      &&          &&  \ar@{-}@(l,l)[dddd]\\
      &\ar@{.}[r]
       &\ar@{-}[r]
                     &&
                                        &&  &\ar@{-}[r]&\ar@{.}[r]
                                                   &&  \\
   {\ivec{\Sigma}^+_{B_1} \setminus \cal{C}^+_{B_1}}
      &&    {\ivec{\Sigma}^+_{B_1} \cap \cal{C}^+_{B_1}}
                      && {\Sigma^+_{B_1}} &&    {\Sigma^-_{B_1}}
                                          && {\vec{\Sigma}^-_{B_2}\cap \cal{C}^-_{B_2}}
                                                   &&  {\vec{\Sigma}^{-}_{B_2} \setminus \cal{C}^-_{B_2}}   \\
      &\ar@{.}[r]& \ar@{-}[r] && && &\ar@{-}[r]&\ar@{.}[r]&&  \\
      && && &{\cat G} & && &&
      \save "2,4"+<-3pt,8pt>."4,8"*[F-:<10pt>]\frm{}\restore
   }\]

In particular, ones sees easily that the finite gluings of \cite{AST} are tilted gluings in our sense. For more examples, we refer the reader to section 7.

\subsection{Representation dimension}

\begin{proposition}
    Let $B_1$ and $B_2$ be respectively a right and a left admissible algebra having representation dimension three and $A = B_1*B_2$. If the slice module $\Sigma^-_{B_2}$ is a direct summand of an Auslander generator for $\mod B_2$, then
     $\repdim A = 3$.
\end{proposition}
\begin{proof}
The proof of \cite{AST}(4.2) applies verbatim.
\end{proof}

\subsection{Induction}

 We define inductively the tilted gluing of $n$ algebras with $n \geq 2$. Assume $B_1* \cdots* B_n$ is defined and assume moreover that it is right admissible while $B_{n+1}$ is left admissible. Then we say that  $A = (B_1* \cdots * B_n)* B_{n+1} = B_1* \cdots * B_n* B_{n+1}$ is a tilted gluing of the $B_i$.

 As above, we denote, for each $i$, by $\Sigma^+_{B_i}$, $\Sigma^-_{B_{i}}$ respectively the left and the right section which define the gluing
\begin{corollary}
    Let $A = B_1* \cdots * B_n$ be a tilted gluing of algebras where  $\repdim B_i = 3$ for all $i$
    and the slice modules of the $\Sigma^-_{B_{i+1}}$ for $ 1 \leq i < n$  are direct summands
    of an Auslander generator for mod $B_{i+1}$. Then $\repdim A = 3$.
\end{corollary}
\begin{proof}
   An easy induction.
\end{proof}

\section{Repetitive categories of tilted algebras}
\subsection{Canonical decomposition}

As seen in the introduction, selfinjective algebras of wild tilted type are orbit algebras of repetitive categories, as introduced in \cite{HW}. For the definition and first properties of repetitive categories, we refer the reader, for instance, to \cite{AST}. We need the following structure result for admissible groups of automorphisms, see \cite{EKS}(3.6).

\begin{proposition}
Let $B$ be a tilted algebra of wild type and $G$ a torsion-free admissible group of automorphisms of $\hat{B}$. Denote by $\nu_{\hat{B}}$ the Nakayama automorphism of $\hat{B}$.  Then $G$ is an infinite cyclic group generated by a strictly positive automorphism of one of the forms
\begin{enumerate}[\indent(a)]
      \item $\sigma\nu^k_{\hat B}$ for a rigid automorphism $\sigma$ and some $k\ge 0$, or
      \item $\mu\varphi^{2k+1}$ for a rigid automorphism $\mu$, a strictly positive automorphism $\varphi$ such that $\varphi^2=\nu_{\hat B}$ and some $k\ge 0$.\qedhere
   \end{enumerate}
\end{proposition}

Let $Q$ be a wild quiver and $B$ a tilted algebra of type $Q$. The most immediate properties of the Auslander-Reiten quiver $\Gamma(\mod \hat{B})$  of the repetitive category $\hat{B}$ of $B$ are summarised in the following theorem (see \cite{EKS}(3.4) and (3.5)).

\begin{thm}
   Let $B$ be a tilted algebra of wild type $Q$.
   Then the Auslander-Reiten quiver of  $\hat B$ is of the form
        $$\Gamma(\mod{\hat B}) = \bigvee_{q\in\Ent}(\cal{X}_q \vee \cal{R}_q)$$
  where, for each $q \in \Ent$,
    \begin{enumerate}[\indent(a)]
        \item $\cal{X}_q$ is an acyclic component whose stable part is of the form $\Ent Q$,
        \item $\cal{R}_q$ is a family of components whose stable part of each is of the form $\Ent\mathbb{A}_{\infty}$,
        \item $\nu_{\hat B}(\cal{X}_q) = \cal{X}_{q+2}$ and $\nu_{\hat B}(\cal{R}_q) = \cal{R}_{q+2}$,
        \item $\cal{X}_q$ separates             $\bigvee_{p < q }(\cal{X}_p \vee \cal{R}_p)~~ $
                            from $~~\cal{R}_q \vee \left(\bigvee_{p > q }(\cal{X}_p \vee \cal{R}_p)\right)$.

    \end{enumerate}
\end{thm}

The description of the previous theorem is known as the \emph{canonical decomposition} of $\Gamma(\mod \hat B)$.

\subsection{Structure of the repetitive category}

Let, as before, $B$ be a tilted algebra of wild type $Q$ and $\Gamma(\mod{\hat B}) = \bigvee_{q\in\Ent}(\cal{X}_q \vee \cal{R}_q)$  the canonical decomposition of $\Gamma(\mod{\hat B})$, as in Theorem 5.1 .  For each $q \in \Ent$, we denote by $\cal{X}^{-}_q$ a fixed right stable full translation subquiver of $\cal{X}_q$ which is closed under successors in $\Gamma(\mod{\hat B})$ and by $\cal{X}^{+}_q$ a fixed left stable full translation subquiver of $\cal{X}_q$ which is closed under predecessors in
$\Gamma(\mod{\hat B})$. We may assume without loss of generality that $\cal{X}^{-}_q = (- \Nat)Q$, $\cal{X}^{+}_q = \Nat Q$,
$\nu_{\hat B}(\cal{X}^{-}_q) = \cal{X}^{-}_{q+2}$ and $\nu_{\hat B}(\cal{X}^{+}_q) = \cal{X}^{+}_{q+2}$ for any $q \in \Ent$. We then have the following theorem \cite{EKS}, Theorem 3.5.

\begin{thm}
   Let $B$ be a tilted algebra of wild type $Q$ and let
      \[ \Gamma(\mod \hat B)=\bigvee_{q\in\Ent}(\cat X_q\vee \cat R_q)\]
    be the canonical decomposition of $\Gamma(\mod \hat B)$. Then, for each $q \in \Ent$,  there exist tilted algebras $B^{-}_q$ and $B^{+}_q$ of type $Q$ such that:
    \begin{enumerate}[\indent(a)]
       \item $B^{-}_q$ is a full convex subcategory of $\hat B$. Moreover,  $B^{-}_q = \End_{\kk Q} T^{-}_q)$, where $T^-_q$ is a tilting $\kk Q$-module without nonzero postprojective direct summand, and $\cal{X}^{-}_q$ is a full translation subquiver of the connecting component $\cal{C}_{T^{-}_q}$ of $\Gamma (\mod B^{-}_q)$ determined by $T^{-}_q$, which is closed under successors in $\cal{C}_{T^{-}_q}$ and consists of torsion $B^{-}_q$-modules.
       \item $B^{+}_q$ is a full convex subcategory of $\hat B$. Moreover, $B^{+}_q = \End_{\kk Q} (T^{+}_q)$, where $T^{+}_q$  is a tilting $\kk Q$-module without nonzero preinjective direct summand and $\cal{X}^{+}_q$ is a full translation subquiver of the connecting component $\cal{C}_{T^{+}_q}$ of $\Gamma (\mod B^{+}_q)$ determined by $T^{+}_q$, which is closed under predecessors in $\cal{C}_{T^{+}_q}$ and consists of torsion-free $B^{+}_q$-modules.
       \item We have $\hat B^{-}_q = \hat B =  \hat B^{+}_q$, $\nu_{\hat B}(B^{-}_q) = B^{-}_{q+2}$ and $\nu_{\hat B}(B^{+}_q) = B^{+}_{q+2}$.
        \item There is a reflection sequence of sinks $i_1,\dotsc,i_r$ of $Q_{B_q^-}$ (possibly empty) such that $B_q^+=S_{i_r}^+\cdots S_{i_1}^+B_q^-$ and
             $B_q=T_{i_r}^+\cdots T_{i_1}^+B_q^-$ is the support algebra of $\cal{R}_q$.
       \item There is a reflection sequence of sinks $j_1,\dotsc,j_s$ of $Q_{B_{q-1}^+}$ (possible empty) such that $B_{q}^-=S_{j_s}^+\cdots S_{j_1}^+B_{q-1}^+$ and
       $D_q=T_{j_s}^+\cdots T_{j_1}^+B_{q-1}^-$ is the support algebra of  $\cat X_q$.

       In particular, $\hat B$ is locally support-finite.
         \end{enumerate}
\end{thm}
Thus, the reader sees that each algebra $B_q$ is a wild quasiserial algebra. Also, the local support finiteness of $\hat{B}$ implies that, if $G$ is an admissible group of automorphisms of $\hat{B}$, then the push-down functor $F_\lambda:\mod \hat B \to \mod (\hat{B}/G)$, associated to the Galois covering $\hat B \to \hat{B}/G$, is dense, see \cite{DS}, \cite{DS2}. In particular, $F_\lambda$ induces an isomorphism between the orbit quiver $\Gamma(\mod \hat{B}) / G$  of $\Gamma(\mod \hat{B})$ under the action of $G$ and the Auslander-Reiten quiver $\Gamma(\mod (\hat{B}/G)$ of $\hat{B}/G$, see \cite{G}, Theorem 3.5.

\section{Proofs of the main theorems}
\subsection{Selfinjective algebras of wild tilted type}

Let $Q$ be a wild quiver. We recall from \cite{EKS} that a  selfinjective algebra $A$ is of \emph{wild tilted type} $Q$ if there exist a tilted algebra $B$ of wild type $Q$ and an admissible infinite cyclic group $G$ of $\kk$-linear automorphisms of $\hat B$ such that $ A = \hat B / G$.

Examples of such algebras are provided by trivial extensions of tilted algebras of type $Q$.

We are now able to prove the main result of the paper.

\begin{thm}
    Let $A$ be a selfinjective algebra of wild tilted type. Then $\repdim A = 3$.
\end{thm}
\begin{proof}
   Because  an algebra is representation-finite if and only if its representation dimension is two, see \cite{A}, and our algebra $A$ is representation-infinite, it suffices to prove that $\repdim A \leq 3$.
   Let $B$ be a tilted algebra of wild tilted type $Q$ and $G$ an infinite cyclic admissible group of automorphisms of $\hat B$ such that $A=\hat B/G$. Then, $G$ is generated by a strictly positive automorphism $g$ of $\hat B$. Also, because of Theorem 5.1,
   $\Gamma(\mod \hat B)$ admits a canonical decomposition
    \[ \Gamma(\mod \hat B)=\bigvee_{q\in\Ent}(\cat X_q\vee \cat R_q).\]
   Furthermore, for each $q\in\Ent$, we have algebras $B_q^-$, $B_q$ and $B_q^+$ which satisfy the conditions of Theorem 5.2.

   Because $G$ also acts on the translation quiver $\Gamma(\mod \hat B)$, there exists $m>0$ such that
   $g(\cal X_q)=\cal X_{q+m}$ and $g(\cal R_q)=\cal R_{q+m}$ for each $q\in\Ent$.
   Then it follows from the definitions of $B^-_q$, $B_q$, $B_q^+$ that we also have
      \[ g(B_q^-)=B_{q+m}^-,\quad g(B_q)=B_{q+m}\quad\text{and } g(B_q^+)=B_{q+m}^+\]
   for each $q\in\Ent$.

   Because of Theorem 5.2, we may choose in the connecting component  $\cal{C}_{T^{-}_q}$ of $B^-_q$ a right section $\Sigma^-_q$
   of wild type $Q$ such that the full translation subquiver $\cat X_q^-$ given by all successors of $\Sigma^-_q$ in $\cal{C}_{T^-_q}$ consists of modules having nonzero restriction to the underlying wild concealed full convex subcategory $C_q$. Moreover, $\cat X^-_q$ is also a full translation subquiver of $\cat X_q$ closed under successors.

   Similarly, we may choose in the connecting component $\cal{C}_{T^+_q}$  of $B^+_q$ a left section $\Sigma^+_q$ of wild type $Q$ such that the full translation subquiver $\cat X_q^+$  given by all predecessors of $\Sigma^+_q$ in $\cal{C}_{T^+_q}$ consists of modules having nonzero restriction to $C_q$. Moreover,  $\cat X_q^+$ is also a full translation subquiver of $\cat X_{q+1}$ closed under predecessors.

   We may assume that $\Sigma^-_q, \Sigma^+_q$ are chosen so that $g\left(\Sigma_q^-\right)=\Sigma_{q+m}^-$ and $g\left(\Sigma_q^+\right)=\Sigma^+_{q+m}$.
   Consequently, $g(\cat X_q^-)=\cat X^-_{q+m}$ and $g(\cat X_q^-)=\cat X^+_{q+m}$ for each $q\in\Ent$.

   For a given $q\in\Ent$, denote by $\cat Y_q$ the finite full translation subquiver of $\cat{X}_q$ consisting of all modules which are successors of $\Sigma_{q-1}^+$
   and predecessors of $\Sigma^+_q$. Then, clearly, every projective-injective in $\cat{X}_q$ lies in $\cat Y_q$.
   Moreover, we have $g(\cat Y_q)=\cat Y_{q+m}$ for any $q$.

   For each $q$, let $M_q$ denote the direct sum of all modules in $\cat Y_q$, all injective $B_q^-$-modules lying in $\cat R_q$ and all projective $\hat B$-modules lying in $\cat R_q$.
   Then, clearly $g(M_q)=M_{q+m}$ for each $q$.

   Finally, we set $M=\bigoplus_{i=0}^{m-1}M_i$.

   Let $F_\lambda:\mod \hat B\to \mod A$ be the push-down functor associated to the Galois covering $F:\hat B\to\hat B/G=A$.
   We claim that $F_\lambda(M)$ is an Auslander generator for $\mod A$.

   First, $F_\lambda(M)$ admits $A$ as a direct summand. Indeed, any indecomposable projective $A$-module is of the form $F_\lambda(P)$, for some
   indecomposable projective $\hat B$-module $P$.
   The definition of $M$ yields an $r\in\Ent$ such that $P$ is a direct summand of ${^{g^r}}\!M$.
   But then $F_\lambda(P)$ is a direct summand of $F_\lambda({^{g^r}\!M})=F_\lambda (M)$.
   We now prove that $\gldim \End M\le 3$, which will complete the proof.

   Let $L$ be an indecomposable $A$-module which is not a direct summand of $F_\lambda(M)$.
   Because the push-down functor is dense, there exist
    $i$ such that $0\le i<m$ and an indecomposable module $X\in(\cat X_i^-\setminus \Sigma_i^-)\vee\cat R_i\vee(\cat X_{i+1}^+\setminus\Sigma_i^+)$ such that $L=F_\lambda(X)$.
   Moreover, if $X\in\cat R_i$, then $X$ is neither a projective $B_i$-module, nor an injective $B_i^-$-module.
   Because of Theorem 3.3 and Corollary 4.3, there exists an $\add M_i$-minimal approximating sequence
   \[ \xymatrix@1{0\ar[r]&U\ar[r]^u&V\ar[r]^v&X\ar[r]&0}\]
   in $\mod\hat B$. Applying the exact functor $F_\lambda$ yields an exact sequence
   \[ \xymatrix@1{0\ar[r]&F_\lambda(U)\ar[r]^{F_\lambda(u)}&F_\lambda(V)\ar[r]^{F_\lambda(v)}&F_\lambda(X)\ar[r]&0}\]
   with $F_\lambda(U)$, $F_\lambda(V)\in\add F_\lambda(M)$.
   We claim that $F_\lambda(v):F_\lambda(V)\to F_\lambda(X)=L$ is an $\add F_\lambda(M)$-approximation.
   Let $h:F_\lambda(M)\to F_\lambda(X)=L$ be a nonzero morphism.
   The push-down functor $F_\lambda:\mod \hat B\to \mod A$ is a Galois covering of module categories.
   In particular, it induces a vector space isomorphism
      \[\Hom_A(F_\lambda(M),F_\lambda(X))\cong \bigoplus_{r\in\Ent}\Hom_{\hat B}({^{g^r}}\!M,X).\]
     Thus, for each $r\in\Ent$, there exists a morphism $f_r:{^{g^r}}\!M\to X$, all but finitely
     many of the $f_r$ being zero, such that $h=\sum_{r\in\Ent}F_\lambda(f_r)$.

     We claim that, for any $r\ge 1$, we have $\Hom_{\hat B}({^{g^r}}\!M,X)=0$.
     Indeed, $X\in\cat X_i\vee\cat R_i\vee\cat X_{i+1}$ for some $i$ with $0\le i<m$.
     On the other hand, for $r\ge 1$, the module ${^{g^r}}\!M$ is a direct sum of modules lying in $\bigvee_{j=0}^{m-1}(\cat X_{j+mr}\vee\cat R_{j+mr})$.
     This establishes our claim.

     Let now $f_r:{^{g^r}}\!M\to X$ be a nonzero morphism in $\mod \hat B$ for some $r\leq 0$.
     Because of Theorem 5.1,  $f_r\,$ factors through a module in $\add M_i$.
     Because $v$ is an $\add M_i$-approximation, there exists a morphism $w_r:{^{g^r}}\!M\to V$ in $\mod \hat B$ such that $f_r=vw_r$.
     Then $F_\lambda(f_r)=F_\lambda(v)F_\lambda(w_r)$ with $F_\lambda(w_r):F_\lambda(M)\to F_\lambda(V)$ because
     $F_\lambda({^{g^r}}\!M)=F_\lambda(M)$.
     Summing up, this yields a morphism $w:F_\lambda(M)\to F_\lambda(V)$ such that $h=F_\lambda(v)w$. Because the existence of $\add F_\lambda(M)$-approximations yields the existence of minimal $\add F_\lambda(M)$-approximations, see 2.2 above , this completes the proof.
\end{proof}

\subsection{Selfinjective algebras with acyclic generalised standard components}

\begin{corollary}
Let $A$ be a selfinjective algebra.
\begin{enumerate} [\indent(a)]
\item If $\Gamma(\mod A)$ has an acyclic generalised standard left stable full translation subquiver which is closed under predecessors, then $\repdim A = 3$.
\item If $\Gamma(\mod A)$ has an acyclic generalised standard right stable full translation subquiver which is closed under successors, then $\repdim A = 3$.
\end{enumerate}
\end{corollary}
\begin{proof}
It follows from \cite{SY} Theorem 5.5 that, under these hypotheses, $A$ is of euclidean or wild tilted type. The statements then follow from \cite{AST} and Theorem 6.1 above.
\end{proof}

Our second main theorem is then.

\begin{thm}
Let $A$ be a selfinjective algebra whose Auslander-Reiten quiver admits an acyclic generalised standard component. Then $\repdim A = 3$.
\end{thm}

\begin{proof}
This follows immediately from the above corollary.
\end{proof}

\section{Examples}

The aim of this section is to present illustrative examples.
The first two describe wild quasiserial algebras
of a different nature.
The third example provides the description of extension and reflection
sequences of algebras associated to a tilted algebra of wild type.

\subsection{Example}
\label{ex:1}

Let $B$ be the algebra given by the quiver $Q$
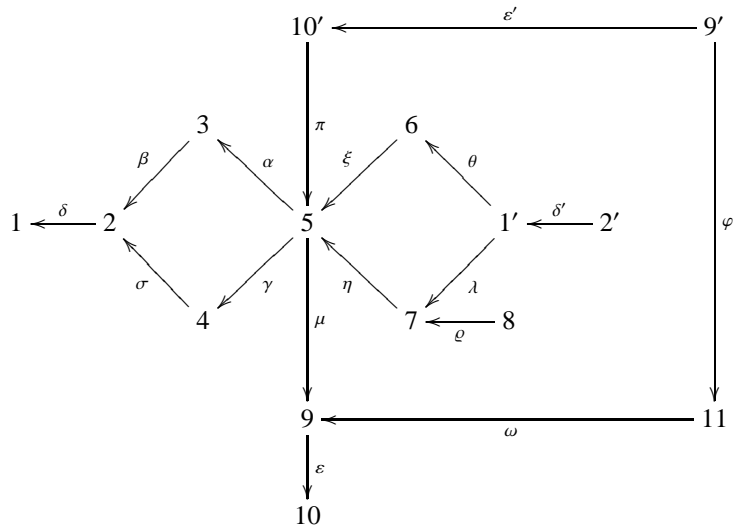
\[
 \xymatrix{
  &&& 10' \ar[dd]^{\pi} &&&& 9' \ar[llll]_{\varepsilon'} \ar[dddd]^{\varphi} \\
  && 3 \ar[ld]_{\beta} && 6 \ar[ld]_{\xi} \\
  1 & 2 \ar[l]_{\delta} && 5 \ar[lu]_{\alpha} \ar[ld]^{\gamma} \ar[dd]^{\mu} &&
      1' \ar[lu]_{\theta} \ar[ld]^{\lambda} & 2' \ar[l]_{\delta'} \\
  && 4 \ar[lu]^{\sigma} && 7 \ar[lu]^{\eta} & 8 \ar[l]^{\varrho} \\
  &&& 9 \ar[d]^{\varepsilon} &&&& 11 \ar[llll]^{\omega} \\
  &&& 10
 }
\]
bound by the relations
$\alpha\beta = \gamma\sigma$,
$\varrho\eta\gamma\sigma = 0$,
$\xi\mu = 0$,
$\eta\mu = 0$,
$\omega\varepsilon = 0$,
$\theta\xi = \lambda\eta$,
$\varepsilon'\pi\mu = \varphi\omega$,
$\pi \alpha = 0$,
$\pi \gamma = 0$,
$\delta'\lambda\eta\gamma\sigma\delta = 0$.
We claim that $B$ is a wild quasiserial algebra.
Let $H$ be the path algebra $\kk \Delta$ of the wild
quiver $\Delta$ of the form
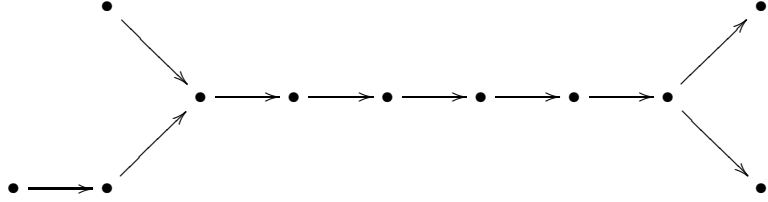
\[
 \xymatrix{
  & \bullet \ar[rd] &&&&&&& \bullet \\
  && \bullet \ar[r] & \bullet \ar[r] & \bullet \ar[r] &
     \bullet \ar[r] & \bullet \ar[r] & \bullet \ar[ru] \ar[rd] \\
  \bullet \ar[r] & \bullet \ar[ru] &&&&&&& \bullet
 }
\]
(of type $\tilde{\tilde{\mathbb{D}}}_9$).
Let $B^-$ be the full convex subcategory of $B$ given
by the objects $1,2,3,4,5,6,7,8,9,10,11$,
$B^+$ the full convex subcategory of $B$ given
by the objects $3,4,5,6,7,8,11,1',2',9',10'$,
and $C$ be the full convex subcategory of $B$ given
by the objects $3,4,5,6,7,8$.
We note that $C$ is a wild hereditary algebra
of type $\tilde{\tilde{\mathbb{D}}}_4$.
Applying \cite{EKS}(2.13) we conclude that
the Auslander-Reiten quiver of $B^-$
has a decomposition
\[
  \Gamma(\mod B^-) = \cal{P}_{B^-} \vee \cal{R}_{B^-} \vee \cal{Q}_{B^-} ,
\]
where
$\cal{P}_{B^-}$ is a postprojective component containing all
indecomposable projective $B^-$-modules and a section of the
form
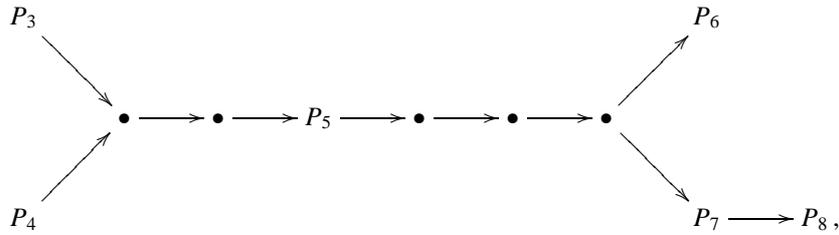
\[
 \xymatrix{
  P_3 \ar[rd] &&&&&&& P_6 \\
  & \bullet \ar[r] & \bullet \ar[r] & P_5 \ar[r] &
     \bullet \ar[r] & \bullet \ar[r] & \bullet \ar[ru] \ar[rd] \\
  P_4 \ar[ru] &&&&&&& P_7 \ar[r] & P_8 \,,\!\!
 }
\]
$\cal{Q}_{B^-}$ is the preinjective component $\cal{Q}_C$
of $\Gamma (\mod C)$,
and $\cal{R}_{B^-}$ consists of infinitely many regular components
of type $\Ent\mathbb{A}_{\infty}$ and a component
$\Clu^-$ of the form
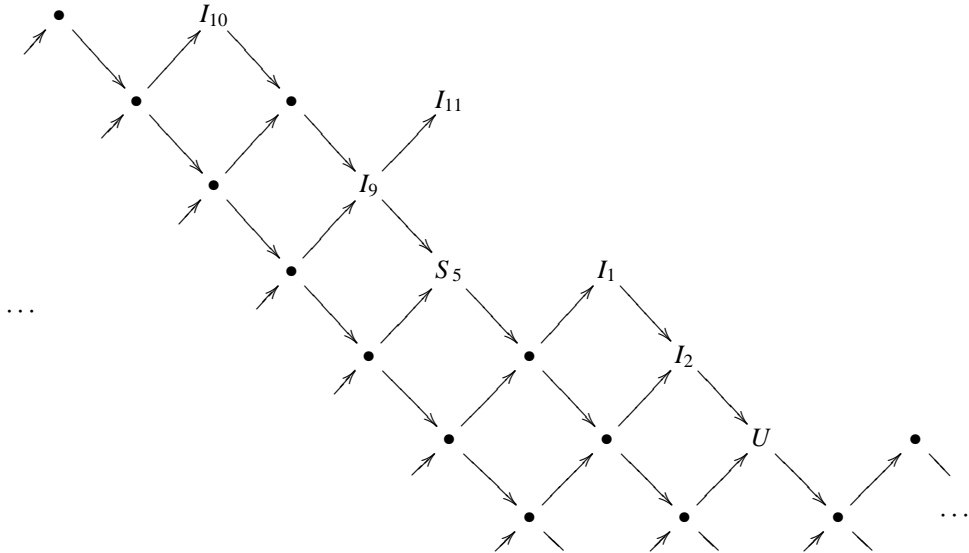
\[
 \xymatrix@C=.5pc@R=.5pc{
  & \bullet \ar[rrdd] &&&& \!\!I_{10}\!\! \ar[rrdd] \\
  \ar[ru] \\
  &&& \bullet \ar[rruu] \ar[rrdd] &&&& \bullet \ar[rrdd] &&&& \!\!I_{11}\!\! \\
  && \ar[ru] \\
  & &&&& \bullet \ar[rruu] \ar[rrdd] &&&& \!I_{9}\! \ar[rruu] \ar[rrdd]  \\
  &&&& \ar[ru] \\
  & &&&&&& \bullet \ar[rruu] \ar[rrdd] &&&& \!S_{5}\! \ar[rrdd] &&&& \!I_{1}\! \ar[rrdd] \\
\save[] *+[c]{\cdots} \restore
  &&&&&& \ar[ru] \\
  & &&&& &&&& \bullet \ar[rruu] \ar[rrdd] &&&& \bullet \ar[rruu] \ar[rrdd] &&&& \!I_{2}\! \ar[rrdd] \\
  &&&& &&&& \ar[ru] \\
  & && &&&& &&&& \bullet \ar[rruu] \ar[rrdd] &&&& \bullet \ar[rruu] \ar[rrdd] &&&&
      \!U \!\ar[rrdd] &&&& \bullet \ar@{-}[rd] \\
  &&&& &&&& && \ar[ru] &&&& &&&& &&&& && \\
  & &&&& &&&& &&&& \bullet \ar[rruu] \ar@{-}[rd] &&&& \bullet \ar[rruu] \ar@{-}[rd] &&&&
      \bullet \ar[rruu] \ar@{-}[rd] &&&
\save[] *+[c]{\cdots} \restore
      \\
  &&&& &&&& &&&& \ar[ru] &&&& \ar[ru] &&&& \ar[ru] && \\
 }
\]
whose stable part, obtained from $\Clu^-$ by deleting
the $\tau_{B^-}$-orbits of the injective $B^-$-modules
$I_1,I_2,I_9,I_{10},I_{11}$,
is of type $\Ent\mathbb{A}_{\infty}$.
In particular, we conclude that
$B^- = \End_H(T^-)$ for a tilting $H$-module
$T^-$ without nonzero postprojective direct summands.
We also note that $B^-$ is a branch coextension of $C$,
using the quasisimple regular $C$-modules $U$ and $S_5$
with the dimension vectors
\begin{align*}
 \bdim U &=
   \begin{array}{*{4}{@{\,}c}@{\,}}1&&1&\vspace{-2.25mm}\\&1&&\vspace{-2.25mm}\\1&&1&0\end{array}
 &
 \bdim S_5 &=
   \begin{array}{*{4}{@{\,}c}@{\,}}0&&0&\vspace{-2.25mm}\\&1&&\vspace{-2.25mm}\\0&&0&0\end{array}
\end{align*}
and the branches
\begin{align*}
  \xymatrix{1 & 2 \ar[l]_{\delta}}
  &&
  \xymatrix{10 & 9 \ar[l]_{\varepsilon} & 11 \ar[l]_{\omega}}
  \mbox{(with $\omega \varepsilon = 0$)}.
\end{align*}
Observe now that $1,2,10,9$ is a reflection sequence
of sinks of $Q_{B^-}$ such that
$B^+ = S^+_9 S^+_{10} S^+_2 S^+_1 B^-$
and
$B = T^+_9 T^+_{10} T^+_2 T^+_1 B^-$.
Further,
the Auslander-Reiten quiver of $B^+$
has a decomposition
\[
  \Gamma(\mod B^+) = \cal{P}_{B^+} \vee \cal{R}_{B^+} \vee \cal{Q}_{B^+} ,
\]
where
$\cal{P}_{B^+}$ is the postprojective component
of $\Gamma (\mod C)$,
$\cal{Q}_{B^+}$ is a preinjective  component containing all
indecomposable injective $B^+$-modules and a section of the
form
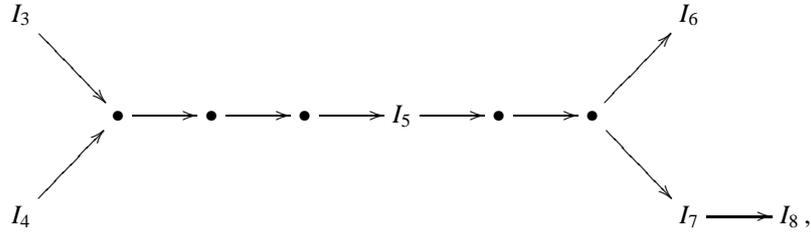
\[
 \xymatrix{
  I_3 \ar[rd] &&&&&&& I_6 \\
  & \bullet \ar[r] & \bullet \ar[r] & \bullet \ar[r] & I_5 \ar[r] &
     \bullet \ar[r] & \bullet \ar[ru] \ar[rd] \\
  I_4 \ar[ru] &&&&&&& I_7 \ar[r] & I_8 \,,\!\!
 }
\]
and $\cal{R}_{B^+}$ consists of infinitely many regular components
of type $\Ent\mathbb{A}_{\infty}$ and a component
$\Clu^+$ of the form
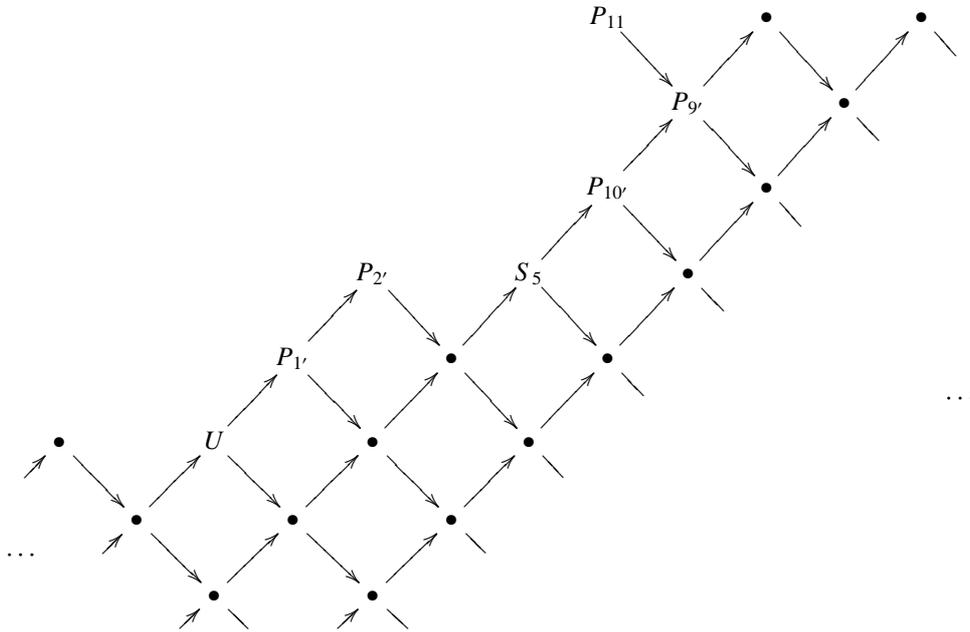
\[
 \xymatrix@C=.5pc@R=.5pc{
  &&&&& &&&&& &&&&& \!\!\!P_{11}\!\!\! \ar[rrdd] &&&& \bullet \ar[rrdd] &&&& \bullet \ar@{-}[rd] \\
  &&&&& &&&&& &&&&& &&&& &&&& & \\
  &&&&& &&&&& &&&&& && \!\!P_{9'}\!\! \ar[rruu] \ar[rrdd] &&&& \bullet \ar[rruu] \ar@{-}[rd] \\
  &&&&& &&&&& &&&&& &&&& &&& \\
  &&&&& &&&&& &&&&& \!\!\!P_{10'}\!\!\! \ar[rruu] \ar[rrdd] &&&& \bullet \ar[rruu] \ar@{-}[rd] \\
  &&&&& &&&&& &&&&& &&&& & \\
  &&&&& &&&& \!\!P_{2'}\!\! \ar[rrdd] &&&& \!\!S_{5}\!\! \ar[rruu] \ar[rrdd] &&&&
    \bullet \ar[rruu] \ar@{-}[rd] \\
  &&&&& &&&&& &&&&& &&& \\
  &&&&& && \!\!P_{1'}\!\! \ar[rruu] \ar[rrdd] &&&& \bullet \ar[rruu] \ar[rrdd] &&&&
    \bullet \ar[rruu] \ar@{-}[rd] \\
  &&&&& &&&&& &&&&& &
  &&&& &&&&
\save[] *+[c]{\cdots} \restore
  \\
  & \bullet \ar[rrdd] && && \!U\! \ar[rruu] \ar[rrdd] &&&&
    \bullet \ar[rruu] \ar[rrdd] &&&& \bullet \ar[rruu] \ar@{-}[rd] \\
  \ar[ru] &&&&& &&&&& &&&& \\
  &&& \bullet \ar[rruu] \ar[rrdd] &&&& \bullet \ar[rruu] \ar[rrdd] &&&&
    \bullet \ar[rruu] \ar@{-}[rd] \\
\save[] *+[c]{\cdots} \restore
  && \ar[ru] &&& &&&&& && \\
  &&&&& \bullet \ar[rruu] \ar@{-}[rd] &&&& \bullet \ar[rruu] \ar@{-}[rd] \\
  &&&& \ar[ru] & &&& \ar[ru] && \\
 }
\]
whose stable part, obtained from $\Clu^+$ by deleting
the $\tau_{B^+}$-orbits of the projective $B^+$-modules
$P_{1'},P_{2'},P_{10'},P_{9'},P_{11}$,
is of type $\Ent\mathbb{A}_{\infty}$.
In particular, we conclude that
$B^+ = \End_H(T^+)$ for a tilting $H$-module
$T^+$ without nonzero preinjective direct summands.
Moreover, $B^+$ is a branch extension of $C$,
using the quasisimple regular
modules
$U$ and $S_5$,
and the branches
\begin{align*}
  \xymatrix{1' & 2' \ar[l]_{\delta'}}
  &&
  \xymatrix{10' & 9' \ar[l]_{\varepsilon'} \ar[r]^{\varphi} & 11}
  .
\end{align*}
Then $B$ is a wild quasiserial algebra whose
Auslander-Reiten quiver
has a decomposition
\[
  \Gamma(\mod B) = \cal{P}_{B^-} \vee \cal{R}_{B} \vee \cal{Q}_{B^+} ,
\]
where $\cal{R}_{B}$ consists of infinitely many regular components
of type $\Ent\mathbb{A}_{\infty}$ and a component
$\Clu$ of the form
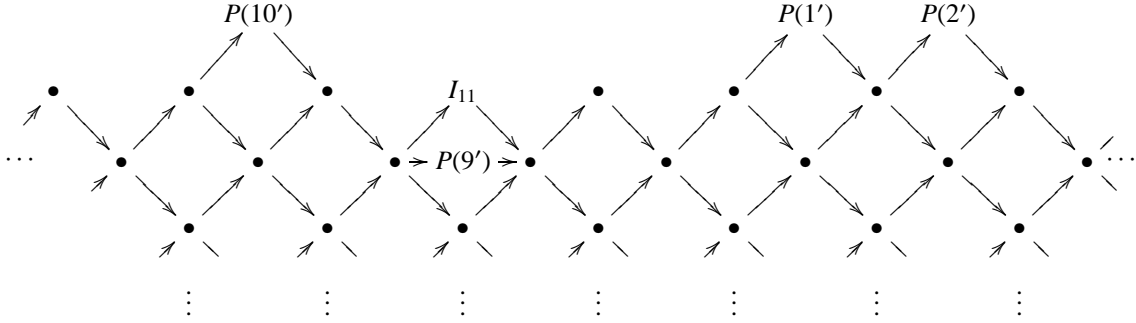
\[
 \xymatrix@C=.35pc@R=.35pc{
  & &&&& && \!\!\!\!\!\!P(10')\!\!\!\!\!\! \ar[rrdd] &&&& &&&& &&&& &&&&
      \!\!\!\!P(1')\!\!\!\! \ar[rrdd] &&&&
      \!\!\!\!P(2')\!\!\!\! \ar[rrdd] \\
  \\
  & \bullet \ar[rrdd] &&&& \bullet \ar[rruu] \ar[rrdd] &&&&
    \bullet \ar[rrdd] &&&& \!\!I_{11}\!\! \ar[rrdd] &&&&
    \bullet \ar[rrdd] &&&& \bullet \ar[rruu] \ar[rrdd] &&&&
    \bullet \ar[rruu] \ar[rrdd] &&&& \bullet \ar[rrdd] \\
  \ar[ru] &&&& &&&& &&&& &&&& &&&& &&&& &&&& &&&& \\
\save[] *+[c]{\cdots} \restore
  &&& \bullet \ar[rruu] \ar[rrdd] &&&& \bullet \ar[rruu] \ar[rrdd] &&&&
    \bullet \ar[rruu] \ar[rrdd] &&
    \save[] *+[c]{P(9')} \ar@{<-}[ll] \ar@{->}[rr] \restore
    && \bullet \ar[rruu] \ar[rrdd] &&&&
    \bullet \ar[rruu] \ar[rrdd] &&&& \bullet \ar[rruu] \ar[rrdd] &&&&
    \bullet \ar[rruu] \ar[rrdd] &&&& \bullet \ar@{-}[ru] \ar@{-}[rd]
& \save[] *+[c]{\cdots} \restore
\\
  && \ar[ru] && &&&& &&&& &&&& &&&& &&&& &&&& &&&& \\
  &&&&& \bullet \ar[rruu] \ar@{-}[rd] &&&& \bullet \ar[rruu] \ar@{-}[rd] &&&&
    \bullet \ar[rruu] \ar@{-}[rd] &&&& \bullet \ar[rruu] \ar@{-}[rd] &&&&
    \bullet \ar[rruu] \ar@{-}[rd] &&&& \bullet \ar[rruu] \ar@{-}[rd] &&&&
    \bullet \ar[rruu] \ar@{-}[rd] \\
  &&&& \ar[ru] &&&& \ar[ru] &&&& \ar[ru] &&&& \ar[ru] &&&& \ar[ru] &&&&
    \ar[ru] &&&& \ar[ru] && \\
  &&&&& \vdots &&&& \vdots &&&&
    \vdots &&&& \vdots &&&&
    \vdots &&&& \vdots &&&&
    \vdots \\
 }
\]
where
$P(1'),P(2'),P(9'),P(10')$
are indecomposable projective-injective $B$-modules
at the vertices $1',2',9',10'$,
respectively.
We note that $\Clu$ contains an indecomposable
injective $B^-$-module $I_{11}$,
which is not an injective $B$-module.

\subsection{Example}
\label{ex:2}

Let $A = \kk Q$ be the path algebra of the wild quiver
\[
    Q: \quad
    \xymatrix{
        1
        &
        2
        \ar@<-.5ex>_{\alpha}[l]
        \ar@<.5ex>^{\beta}[l]
        &
        3
        \ar_{\gamma}[l]
    }
\]
It has been proved in
\cite{SS2}(Example~XVIII.5.18)
that there is a tilting module
$T = T_1 \oplus T_2 \oplus T_3$
in $\mod A$, where $T_1$, $T_2$, $T_3$
are indecomposable regular modules with the dimension vectors
\begin{align*}
 \bdim T_1 &= 2\,3\,0,
 &
 \bdim T_2 &= 6\,9\,1,
 &
 \bdim T_3 &= 1\,2\,0,
\end{align*}
and the associated tilted algebra
$B = \End_A(T)$ is given by the quiver
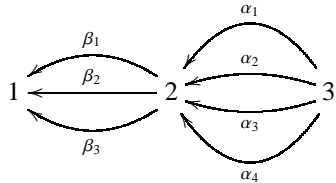
\[
     \xymatrix@C=4pc{
        1
        &
        2
        \ar@<-.5ex>@/_1pc/_{\beta_1}[l]
        \ar_{\beta_2}[l]
        \ar@<.5ex>@/^1pc/^{\beta_3}[l]
        &
        3
        \ar@<-.5ex>@/_2pc/_{\alpha_1}[l]
        \ar@<-.25ex>@/_.5pc/_{\alpha_2}[l]
        \ar@<.25ex>@/^.5pc/^{\alpha_3}[l]
        \ar@<.5ex>@/^2pc/^{\alpha_4}[l]
}
\]
and bound by the relations
\begin{gather*}
 \alpha_1 \beta_2 = 0,
\qquad
 \alpha_1 \beta_3 = 0,
\qquad
 \alpha_4 \beta_1 = 0,
\qquad
 \alpha_4 \beta_2 = 0,
\\
 - \alpha_2 \beta_1
   = \alpha_2 \beta_2
   = - \alpha_3 \beta_2
   = - \alpha_3 \beta_1
   = \alpha_4 \beta_3 ,
\\
 \alpha_1 \beta_1
   = - \alpha_2 \beta_3
   = \alpha_3 \beta_2
   = - \alpha_3 \beta_3 .
\end{gather*}
Moreover, the regular connecting component
$\Clu_T$ of $\Gamma(\mod B)$ determined by $T$
has a canonical section
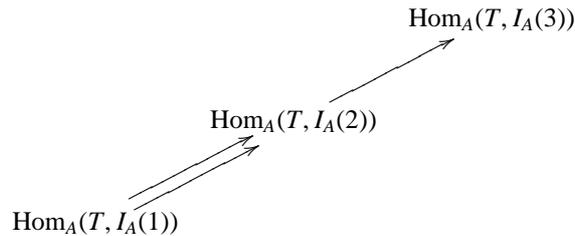
\[
     \xymatrix@C=.5pc{
        && \Hom_A(T,I_A(3)) \\
        & \Hom_A(T,I_A(2)) \ar[ru] \\
        \Hom_A(T,I_A(1)) \ar@<-.5ex>[ru] \ar@<+.5ex>[ru]
     }
\]
of type $Q_A^{\op}$ formed by the indecomposable modules
in $\mod B$ with the dimension vectors
\begin{align*}
 \bdim \Hom_A(T,I_A(1)) &= 2\,6\,1,
 &
 \bdim \Hom_A(T,I_A(2)) &= 3\,9\,2,
 &
 \bdim \Hom_A(T,I_A(3)) &= 0\,1\,0.
\end{align*}
In particular, we conclude that the simple $B$-module
$S_B(2)$ at the vertex $2$ lies in $\cC_T$.
Denote by $C$ the path algebra of the subquiver of $Q_B$
given by the vertices $1$ and $2$, and by $D$ the path algebra
of the subquiver of $Q_B$ given by the vertices $2$ and $3$.
Then $B$ is the one-point extension $B = C[R]$ of $C$
by $R = \rad P_B(3)$, and the one-point coextension
$[U]D$ of $D$ by $U = I_B(1) / \soc I_B(1)$.
Since
$\bdim P_B(3) = 2\,4\,1$
and
$\bdim I_B(1) = 1\,3\,2$,
we conclude that
$R$ is an indecomposable regular $C$-module with
$\bdim R = 2\,4$
and $U$ is an indecomposable regular $D$-module with
$\bdim U = 3\,2$.
This implies that the Auslander-Reiten quiver of $B$
has a decomposition
\[
  \Gamma(\mod B) = \cal{P}_{C} \vee \cal{R}_B^- \vee \cal{C}_{T}
                               \vee \cal{R}_B^+ \vee \cal{Q}_{D} ,
\]
where
$\cal{P}_{C}$ is the postprojective component of $\Gamma (\mod C)$,
$\cal{Q}_{D}$ is the preinjective component of $\Gamma (\mod D)$,
$\cal{R}_B^-$ consists of infinitely many regular components
of type $\Ent\mathbb{A}_{\infty}$ and a component
of the form
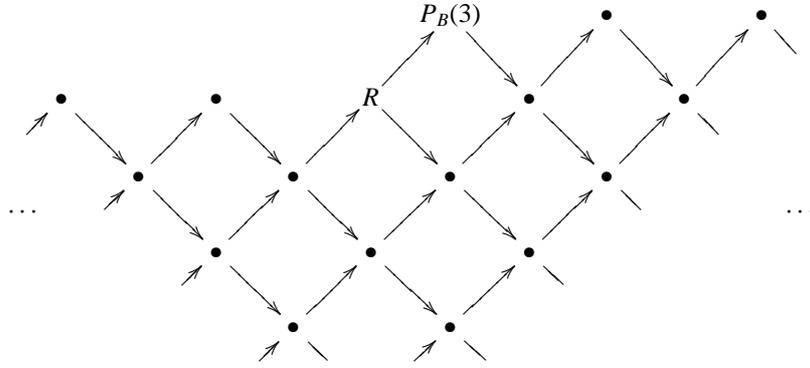
\[
 \xymatrix@C=.5pc@R=.5pc{
  & &&&& &&&& && \!\!\!\!\!P_B(3)\!\!\!\!\! \ar[rrdd] &&&&
    \bullet \ar[rrdd] &&&&  \bullet \ar@{-}[rd] \\
  & &&&& &&&& &&&& &&&& &&& \\
  & \bullet \ar[rrdd] &&&& \bullet \ar[rrdd] &&&&
    \!R\! \ar[rruu] \ar[rrdd] &&&& \bullet \ar[rruu] \ar[rrdd] &&&&
    \bullet \ar[rruu] \ar@{-}[rd] \\
  \ar[ru] &&&& &&&& &&&& &&&& && \\
  &&& \bullet \ar[rruu] \ar[rrdd] &&&& \bullet \ar[rruu] \ar[rrdd] &&&&
    \bullet \ar[rruu] \ar[rrdd] &&&& \bullet \ar[rruu] \ar@{-}[rd] \\
\save[] *+[c]{\cdots} \restore
  && \ar[ru] && &&&& &&&& &&&&
  &&&&
\save[] *+[c]{\cdots} \restore
  \\
  &&&&& \bullet \ar[rruu] \ar[rrdd] &&&& \bullet \ar[rruu] \ar[rrdd] &&&&
    \bullet \ar[rruu] \ar@{-}[rd] \\
  &&&& \ar[ru] && &&&& &&&& \\
  &&& &&&& \bullet \ar[rruu] \ar@{-}[rd] &&&& \bullet \ar[rruu] \ar@{-}[rd] \\
  &&&& && \ar[ru] &&&& \ar[ru] && \\
 }
\]
with stable part of type $\Ent\mathbb{A}_{\infty}$, and
$\cal{R}_B^+$ consists of infinitely many regular components
of type $\Ent\mathbb{A}_{\infty}$ and a component
of the form
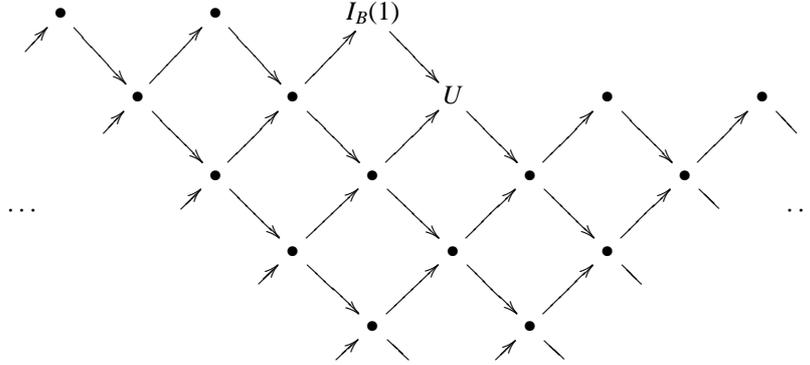
\[
 \xymatrix@C=.5pc@R=.5pc{
  & \bullet \ar[rrdd] &&&& \bullet \ar[rrdd] &&&&
    \!\!\!\!I_B(1)\!\!\!\! \ar[rrdd] \\
  \ar[ru] \\
  &&& \bullet \ar[rruu] \ar[rrdd] &&&& \bullet \ar[rruu] \ar[rrdd] &&&&
    \!U\! \ar[rrdd] &&&& \bullet \ar[rrdd] &&&&
    \bullet \ar@{-}[rd] \\
  && \ar[ru] &&&& &&&& &&&& &&&& && \\
  && &&& \bullet \ar[rruu] \ar[rrdd] &&&& \bullet \ar[rruu] \ar[rrdd] &&&&
    \bullet \ar[rruu] \ar[rrdd] &&&& \bullet \ar[rruu] \ar@{-}[rd] \\
\save[] *+[c]{\cdots} \restore
  && && \ar[ru] && &&&& &&&& &&&&
  &&
\save[] *+[c]{\cdots} \restore
  \\
  && &&&&& \bullet \ar[rruu] \ar[rrdd] &&&& \bullet \ar[rruu] \ar[rrdd] &&&&
    \bullet \ar[rruu] \ar@{-}[rd] \\
  && &&&& \ar[ru] && &&&& &&&& \\
  && &&& &&&& \bullet \ar[rruu] \ar@{-}[rd] &&&& \bullet \ar[rruu] \ar@{-}[rd] \\
  && &&&& && \ar[ru] &&&& \ar[ru] && \\
 }
\]
Consider the algebras
$E = T_1^+ B$ and $F = S_1^+ B$.
Since $T_1^+ B = B[I_B(1)]$,
we conclude that the quiver $Q_E$ of $E$
is of the form.
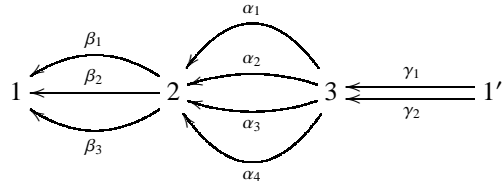
\[
     \xymatrix@C=4pc{
        1
        &
        2
        \ar@<-.5ex>@/_1pc/_{\beta_1}[l]
        \ar_{\beta_2}[l]
        \ar@<.5ex>@/^1pc/^{\beta_3}[l]
        &
        3
        \ar@<-.5ex>@/_2pc/_{\alpha_1}[l]
        \ar@<-.25ex>@/_.5pc/_{\alpha_2}[l]
        \ar@<.25ex>@/^.5pc/^{\alpha_3}[l]
        \ar@<.5ex>@/^2pc/^{\alpha_4}[l]
        &
        1'
        \ar@<-.5ex>_{\gamma_1}[l]
        \ar@<.5ex>^{\gamma_2}[l]
}
\]
Let $H$ be the Kronecker algebra given by the arrows
$\gamma_1$ and $\gamma_2$.
Then $F$ is the one-point extension algebra $D[U]$ and
the one-point coextension algebra $[V]H$ and,
where $V = I_F(2) / \soc I_F(2)$.
We determine the $H$-module $V$.
The Cartan matrix $C_E$ of $E$ is of the form
\[
  \begin{bmatrix}
    1 & 3 & 2 & 1 \\
    0 & 1 & 4 & 3 \\
    0 & 0 & 1 & 2 \\
    0 & 0 & 0 & 1 \\
  \end{bmatrix}
  .
\]
Hence,
$\bdim I_F(2) = 1\,4\,3$,
and then
$\bdim V = 4\,3$.
Therefore, $V$ is the postprojective $H$-module
$\tau_H^{-1} P_H(2)$.
In particular, we conclude that $\Gamma(\mod F)$
has a left stable acyclic component $\cC$ of the form
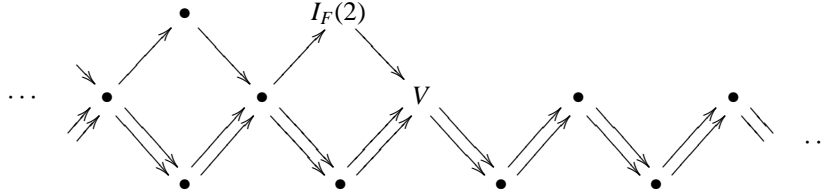
\[
 \xymatrix@C=.5pc@R=.5pc{
  &&&& \bullet \ar[rrdd] &&&&  \!\!\!\!\!I_F(2)\!\!\!\! \ar[rrdd] \\
  & \ar[rd] \\
  \cdots && \bullet \ar[rruu] \ar@<-.5ex>[rrdd] \ar@<+.5ex>[rrdd] &&&&
      \bullet \ar[rruu] \ar@<-.5ex>[rrdd] \ar@<+.5ex>[rrdd] &&&&
      V \ar@<-.5ex>[rrdd] \ar@<+.5ex>[rrdd] &&&&
      \bullet \ar@<-.5ex>[rrdd] \ar@<+.5ex>[rrdd] &&&&
      \bullet \ar@<-.5ex>@{-}[rd] \ar@<+.5ex>@{-}[rd] \\
  & \ar@<-.5ex>[ru] \ar@<+.5ex>[ru] &&&& &&&& &&&& &&&& &&& \cdots \\
  &&&& \bullet \ar@<-.5ex>[rruu] \ar@<+.5ex>[rruu]
      &&&& \bullet \ar@<-.5ex>[rruu] \ar@<+.5ex>[rruu]
      &&&& \bullet \ar@<-.5ex>[rruu] \ar@<+.5ex>[rruu]
      &&&& \bullet \ar@<-.5ex>[rruu] \ar@<+.5ex>[rruu] \\
 }
\]
and consequently $F$ is a tilted algebra of the form
$\End_A(T^*)$ for a tilting module $T^*$ in $\mod A$
without nonzero preinjective direct summands such that $\Clu$
is the connecting component $\Clu_{T^*}$ of $\Gamma(\mod F)$
determined by $T^*$.

Summing up, $E$ is a wild quasiserial algebra
such that $E^- = B$ and $E^+ = F$.
Moreover, the Auslander-Reiten quiver $\Gamma(\mod E)$
has a decomposition
\[
  \Gamma(\mod E) = \cal{P}_{C} \vee \cal{R}_B^- \vee \cal{C}_{T}
                               \vee \cal{R}_E \vee \cal{C}_{T^*} \vee \cal{T}_{H} \vee \cal{Q}_{H} ,
\]
where
$\cal{T}_{H}$ is the $\mathbb{P}_1(\kk)$-family of stable tubes of rank $1$ in $\Gamma(\mod H)$,
$\cal{Q}_{H}$ is the preinjective component of $\Gamma (\mod H)$,
and $\cal{R}_E$ consists of infinitely many regular components
of type $\Ent\mathbb{A}_{\infty}$ and a component
of the form
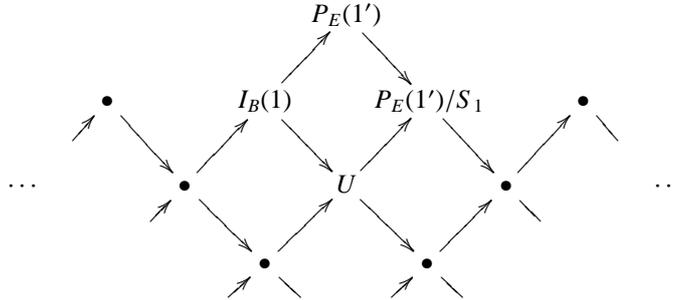
\[
 \xymatrix@C=.5pc@R=.5pc{
  &&&& &&&&
    \!\!\!\!\!\!P_E(1')\!\!\!\!\!\! \ar[rrdd] &&&& \\
  \\
  && \bullet \ar[rrdd] &&&& \!\!\!\!I_B(1)\!\!\!\! \ar[rruu] \ar[rrdd] &&&&
    \!\!\!\!\!\!\!\!\!\!P_E(1')/S_1\!\!\!\!\!\!\!\!\!\!\! \ar[rrdd] &&&& \bullet \ar@{-}[rd] \\
  & \ar[ru] && &&&& &&&& &&&& \\
  \cdots &&&& \bullet \ar[rruu] \ar[rrdd] &&&& \!U\! \ar[rruu] \ar[rrdd] &&&&
    \bullet \ar[rruu] \ar@{-}[rd] &&&& \cdots \\
  &&& \ar[ru] && &&&& &&&& \\
  && &&&& \bullet \ar[rruu] \ar@{-}[rd] &&&& \bullet \ar[rruu] \ar@{-}[rd] \\
  &&& && \ar[ru] &&&& \ar[ru] && \\
 }
\]
with stable part of type $\Ent\mathbb{A}_{\infty}$.

We also note that the opposite algebra $B^{\op}$ of $B$
is a tilted algebra of the form $B^{\op} = \End_{A^{\op}} (D(T))$,
where $D(T)$ is a regular tilting module in $\mod A^{\op}$.
Similarly, $F^{\op}$ is a tilted algebra of the form
$F^{\op} = \End_{A^{\op}} (D(T^*))$,
where $D(T^*)$ is a regular tilting module in $\mod A^{\op}$
without nonzero postprojective direct summands.
Then we conclude that $E^{\op}$ is a wild quasiserial algebra
such that $(E^{\op})^- = F^{\op}$ and $(E^{\op})^+ = B^{\op}$.

\subsection{Example}
\label{ex:3}

Let $B$ be the algebra given by the quiver $Q$ of the form
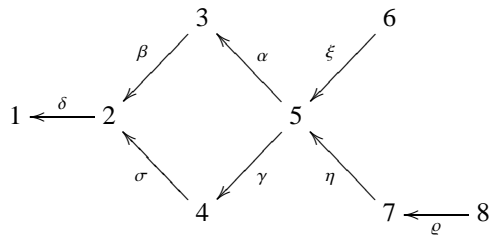
\[
 \xymatrix{
  && 3 \ar[ld]_{\beta} && 6 \ar[ld]_{\xi} \\
  1 & 2 \ar[l]_{\delta} && 5 \ar[lu]_{\alpha} \ar[ld]^{\gamma} \\
  && 4 \ar[lu]^{\sigma} && 7 \ar[lu]^{\eta} & 8 \ar[l]^{\varrho} \\
 }
\]
bound by the relations
$\alpha\beta = \gamma\sigma$
and
$\varrho\eta\gamma\sigma = 0$.
Then $B$ is a tilted algebra of the form $\End_H(T)$,
where $H$ is the path algebra $\kk \Delta$ of the
quiver $\Delta$ of the form
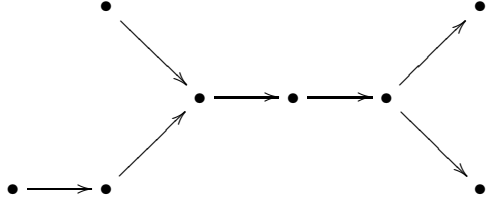
\[
 \xymatrix{
  & \bullet \ar[rd] &&&& \bullet \\
  && \bullet \ar[r] & \bullet \ar[r] & \bullet \ar[ru] \ar[rd] \\
  \bullet \ar[r] & \bullet \ar[ru] &&&& \bullet
 }
\]
and $T$ is a tilting module in $\mod H$ without nonzero
postprojective direct summands (compare Example~\ref{ex:1}).
Moreover, $B$ is a branch coextension of the wild
hereditary algebra $C$ being the full convex subcategory of $B$
given by the objects $3,4,5,6,7,8$,
using the quasisimple regular $C$-module $R$ with
the dimension vector
$\bdim R = \begin{array}{*{4}{@{\,}c}@{\,}}1&&1&\vspace{-2.5mm}\\&1&&\vspace{-2.5mm}\\1&&1&0\end{array}$.
Then $1,2,3,4,5,6,7,8$ is a reflection sequence
of sinks of $Q_{B} = Q$ such that

\begin{itemize}
 \item
  $S^+_2 S^+_1 B$ is given by the quiver $\sigma^+_2 \sigma^+_1 Q$ of the form
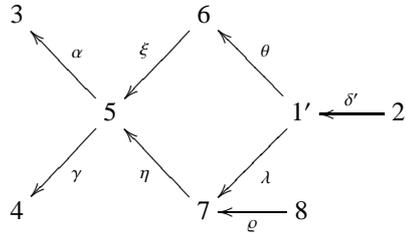
\[
 \xymatrix{
  3 && 6 \ar[ld]_{\xi} \\
  & 5 \ar[lu]_{\alpha} \ar[ld]^{\gamma}
     && 1' \ar[lu]_{\theta} \ar[ld]^{\lambda} & 2' \ar[l]_{\delta'} \\
  4 && 7 \ar[lu]^{\eta} & 8 \ar[l]^{\varrho} \\
 }
\]
  bound by the relation
  $\theta\xi = \lambda\eta$,
  which is a tilted algebra
  of wild type $\tilde{\tilde{\mathbb{D}}}_6$,
  being a branch extension of $C$ using the module $R$.
 \item
  $S^+_4 S^+_3 S^+_2 S^+_1 B$ is given by the quiver
  $\sigma^+_4 \sigma^+_3 \sigma^+_2 \sigma^+_1 Q$ of the form
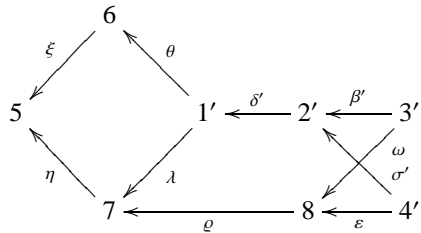
\[
 \xymatrix{
  & 6 \ar[ld]_{\xi} \\
  5 && 1' \ar[lu]_{\theta} \ar[ld]^{\lambda} & 2' \ar[l]_{\delta'}
      & 3' \ar[l]_{\beta'} \ar[ld]^(.25){\omega} \\
  & 7 \ar[lu]^{\eta} && 8 \ar[ll]^{\varrho}
      & 4' \ar[lu]_(.25){\sigma'} \ar[l]^{\varepsilon} \\
 }
\]
  bound by the relations
  $\theta\xi = \lambda\eta$,
  $\beta'\delta'\lambda = \omega\varrho$
  and
  $\sigma'\delta'\lambda = \varepsilon\varrho$,
  which is a tilted algebra
  of wild type $\tilde{\tilde{\mathbb{D}}}_6$,
  being a one-point coextension of the wild concealed full convex
  subcategory $C'$ of $S^+_4 S^+_3 S^+_2 S^+_1 B$
  given by the objects
  $6,7,8,1',2',3',4'$
  by the quasisimple regular $C'$-module $R'$ with $\bdim R'$
  having all coordinates equal $1$.
 \item
  $S^+_5 S^+_4 S^+_3 S^+_2 S^+_1 B$ is given by the quiver
  $\sigma^+_5 \sigma^+_4 \sigma^+_3 \sigma^+_2 \sigma^+_1 Q$ of the form
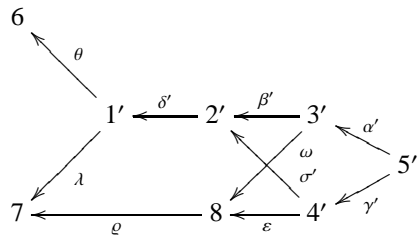
\[
 \xymatrix{
  6 \\
  & 1' \ar[lu]_{\theta} \ar[ld]^{\lambda} & 2' \ar[l]_{\delta'}
      & 3' \ar[l]_{\beta'} \ar[ld]^(.25){\omega}
      & \save[] +<0pc,-1.5pc> *+[c]{5'} \ar[l]_{\alpha'} \ar[ld]^{\gamma'} \restore
      \\
  7 && 8 \ar[ll]^{\varrho}
      & 4' \ar[lu]_(.25){\sigma'} \ar[l]^{\varepsilon} \\
 }
\]
  bound by the relations
  $\beta'\delta'\lambda = \omega\varrho$,
  $\sigma'\delta'\lambda = \varepsilon\varrho$,
  $\alpha'\beta' = \gamma'\sigma'$,
  $\alpha'\omega = \gamma'\varepsilon$,
  which is a tilted algebra
  of wild type $\tilde{\tilde{\mathbb{D}}}_6$,
  being a one-point extension of $C'$
  by the quasisimple regular $C'$-module $R'$.
 \item
  $S^+_8 S^+_7 S^+_6 S^+_5 S^+_4 S^+_3 S^+_2 S^+_1 B$
  is given by the quiver
  $\sigma^+_8 \sigma^+_7 \sigma^+_6 \sigma^+_5 \sigma^+_4 \sigma^+_3 \sigma^+_2 \sigma^+_1 Q$
  of the form
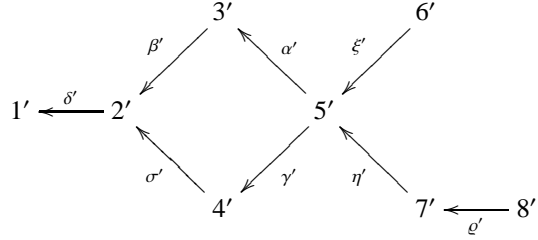
\[
 \xymatrix{
  && 3' \ar[ld]_{\beta'} && 6' \ar[ld]_{\xi'} \\
  1' & 2' \ar[l]_{\delta'} && 5' \ar[lu]_{\alpha'} \ar[ld]^{\gamma'} \\
  && 4' \ar[lu]^{\sigma'} && 7' \ar[lu]^{\eta'} & 8' \ar[l]^{\varrho'} \\
 }
\]
bound by the relations
$\alpha'\beta' = \gamma'\sigma'$
and
$\varrho'\eta'\gamma'\sigma' = 0$,
and hence is isomorphic to $B$.
\end{itemize}

Observe also that
$T^+_2 T^+_1 B$ is a wild quasiserial algebra
with
$(T^+_2 T^+_1 B)^- = B$
and
$(T^+_2 T^+_1 B)^+ = S^+_2 S^+_1 B$,
and
$T^+_5 S^+_4 S^+_3 S^+_2 S^+_1 B$ is a wild quasiserial algebra
with
$(T^+_5 S^+_4 S^+_3 S^+_2 S^+_1 B)^- = S^+_4 S^+_3 S^+_2 S^+_1 B$
and
$(T^+_5 S^+_4 S^+_3 S^+_2 S^+_1 B)^+ = S^+_5 S^+_4 S^+_3 S^+_2 S^+_1 B$.
Moreover,
$T^+_8 T^+_7 T^+_6 T^+_5 T^+_4 T^+_3 T^+_2 T^+_1 B$
is the duplicated algebra
\[
  \bar{B} =
  \begin{bmatrix}
    B & 0 \\
    D(B) & B
  \end{bmatrix}
\]
of $B$.
We also note that $\bar{B}$ is a tilted gluing
\[
 \bar{B}
  = (T^+_2 T^+_1 B) * (T^+_5 S^+_4 S^+_3 S^+_2 S^+_1 B)
     * (S^+_8 S^+_7 S^+_6 S^+_5 S^+_4 S^+_3 S^+_2 S^+_1 B)
     .
\]
The repetitive category $\hat{B}$ of $B$ is given by the quiver
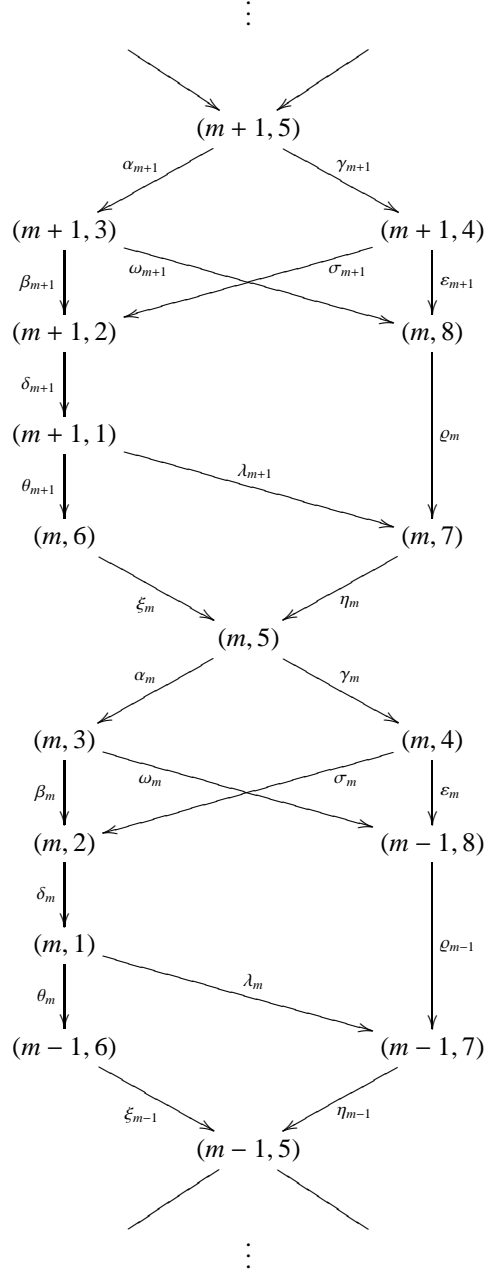
\[
 \xymatrix{
  \save[] +<2pc,-1pc> \ar@{->}[rd] \restore
      &\vdots& \save[] +<-2pc,-1pc> \ar@{->}[ld] \restore \\
  & (m+1,5) \ar[dl]_{\alpha_{m+1}} \ar[dr]^{\gamma_{m+1}} \\
  (m+1,3) \ar[d]_{\beta_{m+1}} \ar[drr]_(.25){\omega_{m+1}} &&
         (m+1,4) \ar[dll]^(.25){\sigma_{m+1}} \ar[d]^{\varepsilon_{m+1}} \\
  (m+1,2) \ar[d]_{\delta_{m+1}} && (m,8) \ar[dd]^{\varrho_m} \\
  (m+1,1) \ar[d]_{\theta_{m+1}} \ar[drr]^{\lambda_{m+1}} \\
  (m,6) \ar[rd]_{\xi_m} && (m,7) \ar[dl]^{\eta_m} \\
  & (m,5) \ar[dl]_{\alpha_m} \ar[dr]^{\gamma_m} \\
  (m,3) \ar[d]_{\beta_m} \ar[drr]_(.25){\omega_m} && (m,4) \ar[dll]^(.25){\sigma_m} \ar[d]^{\varepsilon_m} \\
  (m,2) \ar[d]_{\delta_m} && (m-1,8) \ar[dd]^{\varrho_{m-1}} \\
  (m,1) \ar[d]_{\theta_m} \ar[drr]^{\lambda_m} \\
  (m-1,6) \ar[rd]_{\xi_{m-1}} && (m-1,7) \ar[dl]^{\eta_{m-1}} \\
  & (m-1,5) \\
  \save[] +<2pc,1pc> \ar@{-}[ru] \restore
      &\vdots& \save[] +<-2pc,1pc> \ar@{-}[lu] \restore \\
 }
\]
bound by the relations
\begin{gather*}
 \theta_m \xi_{m-1} = \lambda_m \eta_{m-1},
\qquad
 \beta_m \delta_m\lambda_m = \omega_m\varrho_{m-1},
\qquad
  \sigma_m\delta_m\lambda_m = \varepsilon_m\varrho_{m-1},
\qquad
  \alpha_m\beta_m = \gamma_m\sigma_m,
\qquad
  \alpha_m\omega_m = \gamma_m\varepsilon_m,
\\
  \varrho_m\eta_m\gamma_m\sigma_m = 0,
\qquad
  \delta_{m+1}\theta_{m+1}\xi_m\alpha_m\omega_m = 0,
\qquad
  \beta_{m+1}\delta_{m+1}\theta_{m+1}\xi_m\gamma_m = 0,
\qquad
  \varepsilon_{m+1}\varrho_m\eta_m\alpha_m\beta_m = 0,
\\
  \xi_m\alpha_m\beta_m\delta_m\lambda_m = 0,
\qquad
  \eta_m\gamma_m\sigma_m\delta_m\theta_m = 0,
\qquad
  \delta_{m+1}\theta_{m+1}\xi_m\alpha_m\beta_m\delta_m = 0.
\end{gather*}
We note that every admissible group $G$ of automorphisms
of $\hat{B}$ is an infinite cyclic group $(\nu_{\hat{B}}^r)$,
where $\nu_{\hat{B}}$
is the Nakayama automorphism of $\hat{B}$
and $r$ a positive integer.
Observe that the trivial extension algebra
$\Tensor(B) = \hat{B} / (\nu_{\hat{B}})$
is given by the quiver
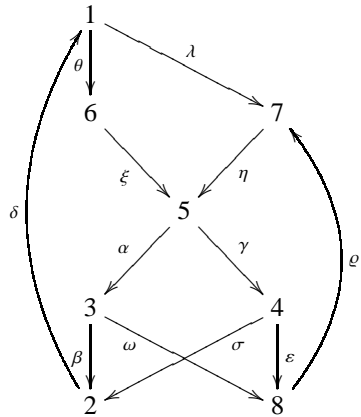
\[
 \xymatrix{
  1 \ar[d]_{\theta} \ar[drr]^{\lambda} \\
  6 \ar[rd]_{\xi} && 7 \ar[dl]^{\eta} \\
  & 5 \ar[dl]_{\alpha} \ar[dr]^{\gamma} \\
  3 \ar[d]_{\beta} \ar[drr]_(.25){\omega} && 4 \ar[dll]^(.25){\sigma} \ar[d]^{\varepsilon} \\
  2 \ar@/^2pc/[uuuu]^{\delta} && 8 \ar@/_2pc/[uuu]_{\varrho} \\
 }
\]
bound by the relations
\begin{gather*}
 \theta \xi = \lambda \eta,
\qquad
 \beta \delta\lambda = \omega\varrho,
\qquad
  \sigma\delta\lambda = \varepsilon\varrho,
\qquad
  \alpha\beta = \gamma\sigma,
\qquad
  \alpha\omega = \gamma\varepsilon,
\\
  \varrho\eta\gamma\sigma = 0,
\qquad
  \delta\theta\xi\alpha\omega = 0,
\qquad
  \beta\delta\theta\xi\gamma = 0,
\qquad
  \varepsilon\varrho\eta\alpha\beta = 0,
\\
  \xi\alpha\beta\delta\lambda = 0,
\qquad
  \eta\gamma\sigma\delta\theta = 0,
\qquad
  \delta\theta\xi\alpha\beta\delta = 0.
\end{gather*}

\subsection*{Acknowledgement}
All the authors were supported by the research grant DEC-2011/02/A/ST1/00216 of the National Science Center Poland.
The first author was also partially supported by the NSERC of Canada, and the third author was also partially supported by CONICET, Argentina.

Research on this paper was carried out while the first and the third author were visiting the second.
They are grateful to him and to the whole representation theory group in Toru\'n for their kind hospitality during their stay.



\end{document}